\documentclass{amsart}

\pdfoutput=1
\usepackage{amssymb}
\usepackage{amsthm}
\usepackage{amsmath}
\usepackage{graphicx}
\usepackage{subcaption}
\usepackage{float}

\title{Multiscale techniques for parabolic equations}
\author[Axel M{\aa}lqvist]{Axel M{\aa}lqvist\textsuperscript{1,2}} 
\author[Anna Persson]{Anna Persson\textsuperscript{1}}

\newtheorem{mydef}{Definition}[section]
\newtheorem{mythm}[mydef]{Theorem}
\newtheorem{mylemma}[mydef]{Lemma}
\theoremstyle{remark}
\newtheorem{myremark}[mydef]{Remark}
\newtheorem{myex}[mydef]{Example}
\newtheorem*{myassump}{Assumptions}
\numberwithin{equation}{section}

\DeclareMathOperator*{\esssup}{ess\,sup}
\DeclareMathOperator*{\essinf}{ess\,inf}
\DeclareMathOperator*{\Span}{span}
\DeclareMathOperator{\supp}{supp}
\DeclareMathOperator{\diam}{diam}

\makeatletter
\newcommand{\enorm}{\@ifstar\@enorms\@enorm}
\newcommand{\@enorms}[1]{%
  \left|\mkern-1.5mu\left|\mkern-1.5mu\left|
   #1
  \right|\mkern-1.5mu\right|\mkern-1.5mu\right|
}
\newcommand{\@enorm}[2][]{%
  \mathopen{#1|\mkern-1.5mu#1|\mkern-1.5mu#1|}
  #2
  \mathclose{#1|\mkern-1.5mu#1|\mkern-1.5mu#1|}
}
\makeatother

\usepackage{hyperref}

\begin{document}

\begin{abstract}
We use the local orthogonal decomposition technique introduced in \cite{Malqvist2014} to derive a generalized finite element method for linear and semilinear parabolic equations with spatial multiscale diffusion coefficient. We consider nonsmooth initial data and a backward Euler scheme for the temporal discretization. Optimal order convergence rate, depending only on the contrast, but not on the variations in the diffusion coefficient, is proven in the $L_\infty(L_2)$-norm. We present numerical examples, which confirm our theoretical findings.
\end{abstract}

\maketitle

\footnotetext[1]{Department of Mathematical Sciences, Chalmers University of Technology and University of Gothenburg SE-412 96 G\"{o}teborg, Sweden. }
\footnotetext[2]{Supported by the Swedish Research Council.}

\section{Introduction}

In this paper we study numerical solutions to a parabolic equation with a highly varying diffusion coefficient. These equations appear, for instance, when modeling physical behavior in a composite material or a porous medium. Such problems are often referred to as \emph{multiscale problems}.
 
Convergence of optimal order of classical finite element methods (FEMs) based on continuous piecewise polynomials relies on at least spatial $H^2$-regularity. More precisely, for piecewise linear polynomials, the error bound depends on $\|u\|_{H^2}$, where $\|u\|_{H^2} \sim \epsilon^{-1}$ if the diffusion coefficient varies on a scale of $\epsilon$. Thus, the mesh width $h$ must fulfill $h<\epsilon$ to achieve convergence. However, this is not computationally feasible in many applications. To overcome this issue, several numerical methods have been proposed, see, for example, \cite{Babuska11}, \cite{Engquist03}, \cite{Hughes98}, \cite{Malqvist2014}, \cite{Owhadi08}, \cite{Owhadi14}, and references therein. In particular, \cite{Owhadi08} and \cite{Owhadi14} consider linear parabolic equations. 

In \cite{Malqvist2014} a generalized finite element method (GFEM) was introduced and convergence of optimal order was proven for elliptic multiscale equations. The method builds on ideas from the variational multiscale method (\cite{Hughes98},\cite{Larson07}), which is based on a decomposition of the solution space into a (coarse) finite dimensional space and a residual space for the fine scales. The method in \cite{Malqvist2014}, often referred to as local orthogonal decomposition, constructs a generalized finite element space where the basis functions contain information from the diffusion coefficient and have support on small vertex patches. With this approach, convergence of optimal order can be proved for an arbitrary positive and bounded diffusion coefficient. Restrictive assumptions such as periodicity of the coefficients or scale separation are not needed. Some recent works (\cite{Henning2014}, \cite{HMP2014}, \cite{Malqvist2013}) show how this method can be applied to boundary value problems, eigenvalue problems, and semilinear elliptic equations. There has also been some recent work on the linear wave equation \cite{Abdulle2014}.

In this paper we apply the technique introduced in \cite{Malqvist2014} to parabolic equations with multiscale diffusion coefficients. For the discretization of the temporal domain we use the backward Euler scheme. Using tools from classical finite element theory for parabolic equations, see, e.g, \cite{larsson92}, \cite{larsson06}, \cite{Thomee2006}, and references therein, we prove convergence of optimal order in the $L_\infty(L_2)$-norm for linear and semilinear equations under minimal regularity assumptions and nonsmooth initial data. The analysis is completed with numerical examples that support our theoretical findings.

In Section~\ref{sec:problem} we describe the problem formulation and the assumptions needed to achieve sufficient regularity of the solution. Section~\ref{sec:LOD} describes the numerical approximation and presents the resulting GFEM. In Section~\ref{sec:error} we prove error estimates and in Section~\ref{sec:semilin} we extend the results to semilinear parabolic equations. Finally, in Section~\ref{sec:numerics} we present some numerical examples.

\section{Problem formulation}\label{sec:problem}
We consider the parabolic problem
\begin{alignat}{2}
\dot{u} - \nabla \cdot (A\nabla u) &= f,&\quad &\text{in } \Omega \times (0,T], \notag \\
u &= 0,& &\text{on } \partial \Omega \times (0,T], \label{parabolic}\\
u(\cdot,0) &= u_0,& &\text{in } \notag \Omega,
\end{alignat}
where $T>0$ and $\Omega$ is a bounded polygonal/polyhedral domain in $\mathbb{R}^d$, $d\leq 3$. We assume $A = A(x)$ and $f=f(x,t)$, that is, the coefficient matrix $A$ does not depend on the time variable. 

We let $H^1(\Omega)$ denote the classical Sobolev space with norm $$\|v\|^2_{H^1(\Omega)} = \|v\|^2_{L_2(\Omega)} + \|\nabla v\|^2_{L_2(\Omega)}$$ and $V=H^1_0(\Omega)$ the space of functions in $H^1(\Omega)$ that vanishes on $\partial \Omega$. We use $H^{-1}(\Omega)=V^\ast$ to denote the dual space to $V$. Furthermore, we use the notation $L_p(0,T;X)$ for the Bochner space with finite norm
\begin{align*}
\|v\|_{L_p(0,T;X)} &= \Big(\int_0^T \|v\|_X^p \, \mathrm{dt} \Big)^{1/p}, \quad 1\leq p<\infty,\\
\|v\|_{L_\infty(0,T;X)} &= \esssup_{0\leq t \leq T} \|v\|_X,
\end{align*}
where $X$ is a Banach space equipped with norm $\| \cdot\|_X$. Here $v\in H^1(0,T;X)$ means $v,\dot{v}\in L_2(0,T;X)$. The dependence on the interval  $[0,T]$ and the domain $\Omega$ is frequently suppressed and we write, for instance, $L_2(L_2)$ for $L_2(0,T;L_2(\Omega))$. Finally, we abbreviate the $L_2$-norm $\|\cdot\| := \|\cdot\|_{L_2(\Omega)}$ and the energy norm, $\enorm{ \cdot }:=\|A^{1/2}\nabla \cdot\|$.

To ensure existence, uniqueness, and sufficient regularity, we make the following assumptions on the data.
\begin{myassump} We assume
\begin{enumerate}
\item[(A1)] $A\in L_\infty(\Omega,\mathbb{R}^{d \times d})$, symmetric, and
\begin{align*}
0<\alpha&:=\essinf_{x \in \Omega} \inf_{v \in \mathbb{R}^d\setminus \{0\}} \frac{A(x) v \cdot v}{v \cdot v},\\
\infty > \beta&:=\esssup_{x \in \Omega} \sup_{v \in \mathbb{R}^d\setminus \{0\}} \frac{A(x) v \cdot v}{v \cdot v},
\end{align*}
\item[(A2)] $u_0 \in L_2$, 
\item[(A3)] $f,\dot{f} \in L_\infty(L_2)$.
\end{enumerate}
\end{myassump}
Throughout this work $C$ denotes constants that may depend on the bounds $\alpha$ and $\beta$ (often through the contrast $\beta/\alpha$), the shape regularity parameter $\gamma$ \eqref{shape} of the mesh, the final time $T$, and the size of the domain $\Omega$, but not on the mesh size parameters nor the derivatives of the coefficients in $A$. The fact that the constant does not depend on the derivatives of $A$ is crucial, since these (if they exist) are large for the problems of interest. This is sometimes also noted as $C$ being independent of \emph{the variations of $A$}.

We now formulate the variational form of problem \eqref{parabolic}. Find $u(\cdot,t) \in V$ such that $u(\cdot,0) = u_0$ and
\begin{align}\label{varform}
(\dot{u},v) + a(u,v) = (f,v), \quad \forall v \in V, \ t\in (0,T],
\end{align}
where $(u,v)=\int_{\Omega}uv$ and $a(u,v) = (A\nabla u,\nabla v)$. 

The following theorem states existence and uniqueness for \eqref{varform}. The proof is based on Galerkin approximations, see, e.g., \cite{Evans2010} and \cite{Lady68}.

\begin{mythm}\label{regularity}
Assume that (A1), (A2), and (A3) holds. Then there exists a unique solution $u$ to \eqref{varform} such that $u \in L_2(0,T;V)$ and $\dot{u}\in L_2(0,T;H^{-1})$.  
\end{mythm}

\section{Numercial approximation}\label{sec:LOD}
In this section we describe the local orthogonal decomposition presented in \cite{Malqvist2014} to define a generalized finite element method for the multiscale problem \eqref{varform}. 

First we introduce some notation. Let \{$\mathcal{T}_h\}_{h>0}$ and \{$\mathcal{T}_H\}_{H>h}$ be families of shape regular triangulations of $\Omega$ where $h_K:= \diam(K)$, for $K\in\mathcal{T}_h$, and $H_K: = \diam(K)$, for $K\in\mathcal{T}_H$. We also define $H:=\max_{K\in \mathcal{T}_H} H_K$ and $h:=\max_{K\in \mathcal{T}_h} h_K$. Furthermore, we let $\gamma>0$ denote the shape regularity parameter of the mesh $\mathcal{T}_H$;
\begin{align}\label{shape}
\gamma:=\max_{K \in \mathcal{T}_H} \gamma_K, \ \text{with} \ \gamma_K:= \frac{\diam B_K}{\diam K}, \ \text{for}\ K\in \mathcal{T}_H,
\end{align}
where $B_K$ is the largest ball contained in $K$. 

Now define the classical piecewise affine finite element spaces
\begin{align*}
V_H &= \{v\in C(\bar{\Omega}): v=0 \text{ on } \partial\Omega, v|_T \text{ is a polynomial of degree} \leq 1, \forall K \in \mathcal{T}_H\},\\
V_h &= \{v\in C(\bar{\Omega}): v=0 \text{ on } \partial\Omega, v|_T \text{ is a polynomial of degree} \leq 1, \forall K \in \mathcal{T}_h\}.
\end{align*}
We let $\mathcal{N}$ denote the interior nodes of $V_H$ and $\varphi_x$ the corresponding hat function for $x\in \mathcal{N}$, such that $\Span(\{\varphi_x\}_{x \in \mathcal{N}})=V_H$. We further assume that $\mathcal{T}_h$ is a refinement of $\mathcal{T}_H$, such that $V_H\subseteq V_h$. 
Finally, we also need the finite element mesh $\mathcal{T}_H$ of $\Omega$ to be of a form such that the $L_2$-projection $P_H$ onto the finite element space $V_H$ is stable in $H^1$-norm, see, e.g., \cite{Bank2014}, and the references therein.

To discretize in time we introduce the uniform discretization 
\begin{align}\label{timedisc}
0=t_0<t_1<...<t_N=T, \text{ where } t_n-t_{n-1} = \tau. 
\end{align}
Let $U_n$ be the approximation of $u(t)$ at time $t=t_n$ and denote $f_n:=f(t_n)$. Using the notation $\bar\partial_t U_n = (U_n-U_{n-1})/\tau$ we now formulate the classical backward Euler FEM; find $U_n\in V_h$ such that
\begin{align}\label{femh}
(\bar\partial_t U_n,v) + a(U_n,v) = (f_n,v), \quad \forall v \in V_h,
\end{align} 
for $n=1,...,N$ and $U_0 \in V_h$ is some approximation of $u_0$. For example, one could choose $U_0=P_h u_0$, where $P_h$ is the $L_2$-projection onto $V_h$. We also define the operator $\mathcal{A}_h: V_h \rightarrow V_h$ by
\begin{align}\label{A_h}
(\mathcal{A}_h v,w) = a(v,w), \quad \forall v, w \in V_h.
\end{align}

The convergence of the classical finite element approximation \eqref{femh} depends on $\|D^2u\|$, where $D^2$ denotes the second order derivatives. If the diffusion coefficient $A$ oscillates on a scale of $\epsilon$ we have $\|D^2u\| \sim \epsilon^{-1}$. Indeed, defining $\mathcal{A}=-\nabla \cdot A \nabla$, elliptic regularity gives
\begin{align*}
\|D^2 u\| &\leq C_1 \|\Delta u\|\leq C_2\|A\Delta u\| \leq C_2\|\nabla \cdot A\nabla u - \nabla A \cdot \nabla u\|
\\&\leq C_2(\|\mathcal{A}u\|+ \|\nabla A \cdot \nabla u\|) \leq C_2(\|\mathcal{A}u\|+ C_A\|\nabla u\|) \leq C_3(1+C_A)\|\mathcal{A}u\|,
\end{align*}
where $C_A$ is a constant that depends on the derivatives (variations) of $A$. This inequality is sharp in the sense that $\mathcal{A}u$ and $\nabla A\cdot \nabla u$ does not cancel in general. The total error is thus $\|u(t_n)-U_n\| \sim (\tau + (h/\epsilon)^2)$, which is small only if $h<\epsilon$.

The purpose of the method described in this paper is to find an approximate solution, let us denote it by $\hat{U}$ for now, in some space $\hat{V} \subset V_h$, such that $\dim {\hat{V}} = \dim{V_H}$, for $H> h$, and the error $\|U_n-\hat{U}_n\|\leq CH^2$. Here $C$ is independent of the variations in $A$ and $\hat{U}_n$ is less expensive to compute than $U_n$. The total error is then the sum of two terms
\begin{align*}
\|u(t_n)-\hat{U}_n\| \leq \|u(t_n)-U_n\| + \|U_n-\hat{U}_n\|,
\end{align*}
where the first term is the error due to the standard FEM approximation with backward Euler discretization in time. This is small if $h$ is chosen small enough, that is, if $h$ resolves the variations of $A$. Hence, we think of $h>0$ as fix and appropriately chosen. Our aim is now to analyze the error $\|U_n-\hat{U}_n\|$. 

We emphasize that $\hat{V}=V_H$ is not sufficient. The total error would in this case be $\|u(t_n)-\hat{U}_n\| \sim (\tau + (H/\epsilon)^2)$, which is small only if $H<\epsilon$.

The next theorem states some regularity results for \eqref{femh}.

\begin{mythm}\label{femreg}
Assume that (A1), (A2), and (A3) holds. Then, for $1\leq n \leq N$, there exists a unique solution $U_n$ to \eqref{varform} such that $U_n \in V_h$. 
Furthermore, if $U_0=0$, then we have the bound
\begin{align}\label{eq:femreg1}
\|\bar\partial_t U_n\| \leq C(\|f\|_{L_\infty(L_2)} + \|\dot{f}\|_{L_\infty(L_2)}),
\end{align}
and, if $f=0$, then
\begin{align}\label{eq:femreg2}
\|\bar\partial_tU_n\| \leq Ct_n^{-1} \|U_0\|,\ n \geq 1, \quad \|\bar\partial_t\bar\partial_tU_n\| \leq Ct_n^{-2} \|U_0\|,\ n\geq 2,
\end{align}
where $C$ depends on $\alpha$ and $T$, but not on the variations of $A$. 
\end{mythm}

\begin{proof}
From \eqref{femh} it follows for $n\geq 2$ that
\begin{align*}
(\bar\partial_t \bar\partial_tU_n,v) + a(\bar\partial_tU_n,v) = (\bar\partial_tf_n,v), \quad \forall v \in V_h,
\end{align*}
and the stability estimate for backward Euler schemes gives
\begin{align*}
\|\bar\partial_tU_n\| \leq \|\bar\partial_tU_1\| + \sum_{j=2}^n\tau\|\bar\partial_tf_j\|.
\end{align*}
From \eqref{femh} we have, since $U_0=0$, $\|\bar\partial_tU_1\| \leq \|f_1\|$. Finally, using the inequality
\begin{align*}
\sum_{j=2}^n\tau\|\bar\partial_tf_j\| \leq \sum_{j=2}^n\max_{t_{j-1}\leq \xi\leq t_j}\tau\|\dot{f}(\xi)\| \leq C\|\dot{f}\|_{L_\infty(L_2)},
\end{align*} 
we deduce \eqref{eq:femreg1}.

For the bound \eqref{eq:femreg2} we refer to \cite[Theorem 7.3]{Thomee2006}.
\end{proof}

\subsection{Orthogonal decomposition}
In this section we describe the orthogonal decomposition which defines the GFEM space denoted $\hat{V}$ in the discussion above. We refer to \cite{Malqvist2014} and \cite{Malqvist2013} for details. 

For the construction of the GFEM space we use the (weighted) Cl{\'e}ment interpolation operator introduced in \cite{Carstensen1999}, $\mathfrak{I}_{H}\colon V_h \rightarrow V_{H}$ defined by
\begin{align}\label{clement}
\mathfrak{I}_H v = \sum_{x\in \mathcal{N}}(\mathfrak{I}_H v)(x)\varphi_x, \quad \text{where} \quad (\mathfrak{I}_H v)(x) : = \frac{\int_{\Omega} v \varphi_x}{\int_{\Omega} \varphi_x}.
\end{align}
For this interpolation operator the following result is proved \cite{Carstensen1999}
\begin{align}\label{interp}
H^{-1}_K\|v-\mathfrak{I}_H v\|_{L_2(K)} + \|\nabla(v-\mathfrak{I}_H v)\|_{L_2(K)} \leq C \|\nabla v\|_{L_2(\bar\omega_K)}, \forall v \in V,
\end{align}
where $\bar\omega_K :=\cup \{\bar K \in \mathcal{T}_H: \bar K \cap K \neq \emptyset \}$ and $C$ depends on the shape regularity $\gamma$.

Let $V^\mathrm{f} = \{v \in V_h: \mathfrak{I}_H v = 0\}$ be the kernel of the Cl{\'e}ment interpolation operator \eqref{clement}. This space contains all fine scale features not resolved by $V_H$. The space $V_h$ can then be decomposed into $V_h = V_H \oplus V^\mathrm{f} $, where $v\in V_h$ can be written as a sum $v=v_H+v^\mathrm{f}$, with $v_H\in V_H$, $v^\mathrm{f} \in V^\mathrm{f}$, and $(v_H,v^\mathrm{f})=0$. 

Now define the orthogonal projection $R^\mathrm{f}\colon V_h \rightarrow V^\mathrm{f}$ by
\begin{align*}
a(R^\mathrm{f}v, w) = a( v,  w) \quad \forall w \in V^\mathrm{f}, \ v \in V_h.
\end{align*}
Using this projection we define the GFEM space, also referred to as the multiscale space,
\begin{align*}
V^{\mathrm{ms}} :=V_H-R^\mathrm{f}V_H,
\end{align*}
which leads to another orthogonal decomposition $V_h = V^{\mathrm{ms}} \oplus V^\mathrm{f}$. Hence any function $v\in V_h$ has a unique decomposition $v=v^\mathrm{ms} + v^\mathrm{f}$, with $v^\mathrm{ms}\in V^\mathrm{ms}$ and $v^\mathrm{f} \in V^\mathrm{f}$, with $a(v^\mathrm{ms},v^\mathrm{f})=0$. 

To define a basis for $V^\mathrm{ms}$ we need to find the projection $R^\mathrm{f}$ of the nodal basis function $\varphi_x\in V_H$. Let this projection be denoted $\phi_x$, so that $\phi_x\in V^\mathrm{f}$ satisfies the (global) corrector problem
\begin{align}\label{corrector}
a(\phi_x,w) = a(\varphi_x,w), \quad \forall w\in V^\mathrm{f}.
\end{align}
A basis for the multiscale space $V^\mathrm{ms}$ is thus given by
\begin{align*}
\{\varphi_x-\phi_x\colon  x\in \mathcal{N}\}.
\end{align*}

We also introduce the projection $R^\mathrm{ms}\colon V_h\rightarrow V^\mathrm{ms}$, defined by
\begin{align}\label{proj}
a(R^\mathrm{ms} v, w) = a(v,w), \quad \forall w \in V^{\mathrm{ms}}, \ v \in V_h.
\end{align}
Note that $R^\mathrm{ms} = I-R^\mathrm{f}$. For $R^\mathrm{ms}$ we have the following lemma, based on the results in \cite{Malqvist2014}.

\begin{mylemma}\label{projconv}
For the projection $R^\mathrm{ms}$ in \eqref{proj} and $v\in V_h$ we have the error bound
\begin{align}
\|v-R^\mathrm{ms} v\| &\leq CH^2\|\mathcal{A}_h v\|, \quad v \in V_h,\label{eq:projconv2}
\end{align}
where $C$ depends on $\alpha$ and $\gamma$, but not on the variations of $A$.
\end{mylemma}

\begin{proof}
Define the following elliptic auxiliary problem: find $z \in V_h$ such that
\begin{align*}
a(z,w) = (v - R^\mathrm{ms} v,w), \quad \forall w \in V_h.
\end{align*}
In \cite[Lemma~3.1]{Malqvist2014} it was proven that the solution to an elliptic equation of the form
\begin{align*}
a(u,w)=(g,w), \quad \forall w \in V_h,
\end{align*}
satisfies the error estimate
\begin{align*}
\enorm{u-R^\mathrm{ms} u} \leq CH\|g\|,
\end{align*}
where $C$ depends on $\gamma$ and $\alpha$, but not on the variations of $A$. Hence, we have the following bound for $z$,
\begin{align*}
\enorm{z-R^\mathrm{ms} z} \leq CH\|v-R^\mathrm{ms} v\|.
\end{align*}
Furthermore, we note that $v-R^\mathrm{ms} v \in V_h$ and
\begin{align*}
\|v-R^\mathrm{ms} v\|^2 &= (v - R^\mathrm{ms} v,v - R^\mathrm{ms} v) = a(z,v - R^\mathrm{ms} v) \\&= a(z - R^\mathrm{ms} z,v - R^\mathrm{ms} v)\leq \enorm{z-R^\mathrm{ms} z}\;\enorm{v - R^\mathrm{ms} v}.
\end{align*}
Now, since $a(v,w)=(\mathcal{A}_hv,w)$, we get  $\enorm{v-R^\mathrm{ms} v}\leq CH\|\mathcal{A}_h v\|$ and \eqref{eq:projconv2} follows.
\end{proof}

In particular, if $U_n$ is the solution to \eqref{femh}, then \eqref{eq:projconv2} gives 
\begin{align*}
\|U_n-R^\mathrm{ms} U_n\| &\leq CH^2\|P_hf_n-\bar\partial_tU_n\|, \quad n\geq 1,\\ 
\|\bar\partial_tU_n-R^\mathrm{ms}\bar\partial_t U_n\| &\leq CH^2\|P_h\bar\partial_t f_n-\bar\partial_t\bar\partial_t U_n\|, \quad n\geq 2.
\end{align*}

The result in Lemma~\ref{projconv} should be compared with the error of the classical Ritz projection $R_h:V\rightarrow V_h$ defined by $a(R_hv,w)=a(v,w)$, $\forall w \in V_h$. Using elliptic regularity estimates, one achieves 
\begin{align*}
\|R_hv-v\|\leq Ch^2\|D^2v\|\leq Ch^2\|\mathcal{A}v\|,
\end{align*}
which is similar to the result in Lemma~\ref{projconv}. However, in this case, $C$ depends on the variations of $A$, as we noted in the discussion in the beginning of this section. This is avoided using the $R^\mathrm{ms}$-projection, since the constant in Lemma~\ref{projconv} does not depend on the variations of $A$. 

Now let $P^\mathrm{ms}$ denote the $L_2$-projection onto $V^{\mathrm{ms}}$ and define the corresponding GFEM to problem \eqref{femh}; find $U^{\mathrm{ms}}_n \in V^{\mathrm{ms}}$ such that $U^{\mathrm{ms}}_0 = P^\mathrm{ms} U_{0}$ and
\begin{align}\label{gfem}
(\bar\partial_t U^{\mathrm{ms}}_n,v) + a(U^{\mathrm{ms}}_n,v) = (f_n,v), \quad \forall v \in V^{\mathrm{ms}},
\end{align}
for $n=1,...,N$. Furthermore, we define the operator $\mathcal{A}^\mathrm{ms}: V^\mathrm{ms} \rightarrow V^\mathrm{ms}$ by
\begin{align}\label{A^ms}
(\mathcal{A}^\mathrm{ms} v,w) = a(v,w), \quad \forall v, w \in V^\mathrm{ms}.
\end{align}

\subsection{Localization}\label{subsec:localization}
Since the corrector problems \eqref{corrector} are posed in the fine scale space $V^\mathrm{f}$ they are computationally expensive to solve. Moreover, the correctors $\phi_x$ generally have global support, which destroys the sparsity of the resulting linear system \eqref{gfem}. However, as shown in \cite{Malqvist2014}, $\phi_x$ decays exponentially fast away from $x$. This observation motivates a localization of the corrector problems to smaller patches of coarse elements. Here we use a variant presented in \cite{Henning2014}, which reduces the required size of the patches. 

We first define the notion of patches and their sizes. For all $K \in \mathcal{T}_H$ we define $\omega_k(K)$ to be the patch of size $k$, where
\begin{align*}
&\omega_0(K) := K,\\
&\omega_k(K) := \cup \{\bar{K}\in \mathcal{T}_H: \bar{K}\cap \omega_{k-1}(K) \neq \emptyset\}, \quad k=1,2,...
\end{align*}
Moreover, we define $V^\mathrm{f}(\omega_k(K)) := \{w \in V^\mathrm{f}: \supp(w)\subset \omega_k(K)\}$.

Now define the operator $R^\mathrm{f}_K\colon V_h \rightarrow V^\mathrm{f}$ by
\begin{align*}
\int_\Omega A\nabla R^\mathrm{f}_Kv\cdot \nabla w = \int_K A\nabla v \cdot\nabla w, \quad \forall v \in V_h,\ w \in V^\mathrm{f},
\end{align*}
and note that $R^\mathrm{f} := \sum_{K \in \mathcal{T}_H}R^\mathrm{f}_K$. We now localize the operator $R^\mathrm{f}_K$ by defining $R_{K,k}^\mathrm{f}\colon V_h \rightarrow V^\mathrm{f}(\omega_k(K))$ through
\begin{align*}
\int_{\omega_k(K)}A\nabla R_{K,k}^\mathrm{f}v\cdot \nabla w = \int_K A\nabla v \cdot \nabla w, \quad \forall v \in V_h,\ w \in V^\mathrm{f}(\omega_k(K)),
\end{align*}
and we define $R_{k}^\mathrm{f} := \sum_{K \in \mathcal{T}_H}R_{K,k}^\mathrm{f}$. Hence we can, for each nonnegative integer $k$, define a localized multiscale space
\begin{align*}
V^\mathrm{ms}_k:=V_H - R^\mathrm{f}_kV_H.
\end{align*} 
Here the basis is given by $\{\varphi_x-\phi_{k,x}: x \in \mathcal{N}\}$, where $\phi_{k,x} = R_{k}^\mathrm{f}\varphi_x$ is the localized version of $\phi_x$. The procedure of decomposing $V_h$ into the orthogonal spaces $V^\mathrm{ms}$ and $V^\mathrm{f}$ together with the localization of $V^\mathrm{ms}$ to $V^\mathrm{ms}_k$ is referred to as \emph{local orthogonal decomposition}.

The following lemma follows from Lemma~3.6 in \cite{Henning2014}.  
\begin{mylemma}\label{correxp}
There exists a constant $0 < \mu < 1$ that depends on the contrast $\beta/\alpha$ such that
\begin{align*}
\enorm{R^\mathrm{f}v-R^\mathrm{f}_kv} \leq Ck^{d/2}\mu^{k}\enorm{ v }, \quad \forall v \in V_h,
\end{align*}
where $C$ depends on $\beta$, $\alpha$, and $\gamma$, but not on the variations of $A$.
\end{mylemma}

Now let $R^\mathrm{ms}_k\colon V_h \rightarrow V^\mathrm{ms}_k$ be the orthogonal projection defined by
\begin{align} \label{lodproj}
a(R^\mathrm{ms}_k v, w) = a(v, w), \quad \forall w \in V^\mathrm{ms}_k.
\end{align}
Next lemma is a consequence of Theorem~3.7 in \cite{Henning2014} and estimates the error due to the localization procedure. 

\begin{mylemma}\label{lodprojconv}
For the projection $R^\mathrm{ms}_k$ in \eqref{lodproj} we have the bound
\begin{align}\label{eq:lodprojconv} 
\| v -R^\mathrm{ms}_k v\| &\leq C(H+k^{d/2}\mu^{k})^2\|\mathcal{A}_h v\|, \quad \forall v \in V_h.
\end{align}
Here $C$ depends on $\beta$, $\alpha$, and $\gamma$, but not on the variations of $A$.
\end{mylemma}

\begin{proof}
The proof is similar to the proof of Lemma~\ref{projconv}. 
Let $z\in V_h$ be the solution to the elliptic dual problem
\begin{align*}
a(z,w)=(v - R^\mathrm{ms}_k v,w), \quad \forall w \in V_h,
\end{align*}
which gives
\begin{align*}
\|v-R^\mathrm{ms}_k v\|^2 &= (v-R^\mathrm{ms}_k v, v-R^\mathrm{ms}_k v) = a(z-R^\mathrm{ms}_k z,v-R^\mathrm{ms}_k v)\\ &\leq \enorm{z-R^\mathrm{ms}_k z} \enorm{v-R^\mathrm{ms}_kv}.
\end{align*}
It follows from Theorem~3.7 in \cite{Henning2014} that there exists a constant $C$ depending on $\beta$, $\alpha$, and $\gamma$, such that $\enorm{z-R^\mathrm{ms}_k z}\leq C(H+k^{d/2}\mu^k)\|v-R^\mathrm{ms}_k v\|$, with $\mu$ as in Lemma~\ref{correxp}. Since $(\mathcal{A}_hv,w)=a(v,w)$ we get  $\enorm{v-R^\mathrm{ms}_kv}\leq CH\|\mathcal{A}_hv\|$ and \eqref{eq:lodprojconv} follows.
\end{proof}

We are now ready to formulate the localized version of \eqref{gfem} by replacing $V^\mathrm{ms}$ by $V^\mathrm{ms}_k$. The localized GFEM formulation reads; find $U^{\mathrm{ms}}_{k,n} \in V^{\mathrm{ms}}_k$ such that $U^{\mathrm{ms}}_{k,0} = P^\mathrm{ms}_k U_{0}$ and
\begin{align}\label{gfemlod}
(\bar\partial_t U^{\mathrm{ms}}_{k,n},v) + a(U^{\mathrm{ms}}_{k,n},v) = (f_n,v), \quad \forall v \in V^{\mathrm{ms}}_k,
\end{align}
for $n=1,...,N$, where $P^\mathrm{ms}_k$ is the $L_2$-projection onto $V^\mathrm{ms}_k$. We also define the operator $\mathcal{A}^\mathrm{ms}_k: V^\mathrm{ms}_k \rightarrow V^\mathrm{ms}_k$, a localized version of \eqref{A^ms}, by
\begin{align}\label{A^ms_k}
(\mathcal{A}^\mathrm{ms}_k v,w) = a(v,w), \quad \forall v, w \in V^\mathrm{ms}_k.
\end{align}

Next lemma states some important properties of the operators $\mathcal{A}_h$, $\mathcal{A}^\mathrm{ms}$, and $\mathcal{A}^\mathrm{ms}_k$.
\begin{mylemma}\label{operatorsA}
The operators $\mathcal{A}_h$, $\mathcal{A}^\mathrm{ms}$, and $\mathcal{A}^\mathrm{ms}_k$ are self-adjoint and positive definite. Furthermore, the following bound hold
\begin{align*}
\|(\mathcal{A}^\mathrm{ms}_k)^{-1/2}P^\mathrm{ms}_kf\| \leq C \|f\|_{H^{-1}}, \quad \forall f \in L_2,
\end{align*}
where $C$ depends on $\alpha$.
\end{mylemma}
\begin{proof}
The fact that the operators are self-adjoint and positive definite follows from the assumptions on $A$ (A1). A proof of the bound can be found in \cite{larsson06}.
\end{proof}

We also define the solution operator $E^\mathrm{ms}_{k,n}=((I+\tau\mathcal{A}^\mathrm{ms}_k)^{-1})^n$, such that the solution to \eqref{gfemlod}, with $f=0$, can be expressed as $U^\mathrm{ms}_{k,n} = E^\mathrm{ms}_{k,n}U^\mathrm{ms}_{k,0}$. For this operator we have estimates similar to \eqref{eq:femreg2}.
\begin{mylemma}\label{esthomogeneous_disc} For $l=0,1$, and $v\in L_2$, we have
\begin{align*}
\|\bar \partial_t^l E^\mathrm{ms}_{k,n}P^\mathrm{ms}_kv\| \leq Ct_n^{-l}\|v\|, \quad n\geq l, \quad \enorm{E^\mathrm{ms}_{k,n}P^\mathrm{ms}_kv} \leq Ct_n^{-1/2}\|v\|,\quad n\geq 1,
\end{align*}
where $C$ depends on the constant $C_k=\sup_{s>0}s^ke^{-s}$, $\beta$, and $\alpha$.
\end{mylemma}
\begin{proof}
The operator $\mathcal{A}^\mathrm{ms}_k$ is self-adjoint, positive definite, and defined on the finite dimensional space $V^\mathrm{ms}_k$. Thus, there exist a finite number of positive eigenvalues $\{\lambda_i\}_{i=1}^M$ and corresponding orthogonal eigenvectors $\{\varphi_i\}_{i=1}^M$ such that $\Span\{\varphi_i\}=V^\mathrm{ms}_k$. We refer to \cite{Malqvist2013} for a further discussion on the eigenvalues to the operator $A^\mathrm{ms}_k$. 

It follows that $E^\mathrm{ms}_{k,n}v$ can be written as
\begin{align*}
E^\mathrm{ms}_kv = \sum_{i=1}^M \frac{1}{(1+\tau \lambda_i)^n}(v,\varphi_i)\varphi_i,
\end{align*}
and the estimates now follows from \cite[Lemma 7.3]{Thomee2006}.
\end{proof}

\section{Error analysis}\label{sec:error}
In this section we derive error estimates for the local orthogonal decomposition method introduced in Section~\ref{sec:LOD}. The localized GFEM solution \eqref{gfemlod} is compared to the classical FEM solution \eqref{femh}, which leads to a setting where the initial data is not smooth, since $U_0 \in V_h$ only. This leads to error bounds which are non-uniform in time, but of optimal order for a fix time $t_n>0$. The same phenomenon appears in classical finite element analysis for equations with nonsmooth initial data, see \cite{Thomee2006} and references therein. The error analysis in this section is carried out by only taking the $L_2$-norm of $U_{0}$, which allows $u_0 \in L_2$. If we, for instance, choose $U_0=P_hu_0$, then $\|U_0\|\leq \|u_0\|$.  

\begin{mythm}\label{femconv}
Let $U_n$ be the solution to \eqref{femh} and $U^{\mathrm{ms}}_{k,n}$ the solution to \eqref{gfemlod}. Then, for $1\leq n\leq N$,
\begin{align*}
\|U^{\mathrm{ms}}_{k,n}-U_n\| &\leq C\Big(1+\log\frac{t_n}{\tau}\Big)(H+k^{d/2}\mu^k)^2\big(t_n^{-1}\|U_{0}\| +  \|f\|_{L_\infty(L_2)} + \|\dot{f}\|_{L_\infty(L_2)}\big),
\end{align*}
where $C$ depends on $\beta$, $\alpha$, $\gamma$, and $T$, but not on the variations of $A$.
\end{mythm}

The proof of Theorem~\ref{femconv} is divided into two lemmas. The first covers the homogeneous case, $f=0$, and the second covers the nonhomogeneous case with vanishing intital data $u_0=0$. To study the error in the homogeneous case we use techniques similar to the classical finite element analysis of problems with nonsmooth initial data, see \cite{Thomee2006} and the references therein. 

Define $T_h=\mathcal{A}^{-1}_hP_h$ and $T^\mathrm{ms}_k=(\mathcal{A}^\mathrm{ms}_k)^{-1}P^\mathrm{ms}_k$. With this notation the solution to the parabolic problem \eqref{femh}, with $f=0$, can be expressed as $T_h\bar\partial_t{U}_n + U_n=0$. Similarly, the solution to \eqref{gfemlod}, with $f=0$, can be expressed as $T^\mathrm{ms}_k\bar\partial_t{U}^\mathrm{ms}_{k,n} + U^\mathrm{ms}_{k,n}=0$. Note that $T^\mathrm{ms}_k$ is self-adjoint and positive semidefinite on $L_2$, and that $T^\mathrm{ms}_k = R^\mathrm{ms}_k T_h$. 

Now, let $e_n=U^\mathrm{ms}_{k,n}-U_n$, where $e_n$ solves the error equation 
\begin{align}\label{erroreq}
T^\mathrm{ms}_k\bar\partial_t {e_n}+e_n=-U_n-T^\mathrm{ms}_k\bar\partial_t U_n=(T_h-T^\mathrm{ms}_k)\bar\partial_tU_n= (R^\mathrm{ms}_k-I) U_n=:\rho_{n},
\end{align}
for $n=1,...,N$ with $T^\mathrm{ms}_ke_0=0$, since $U^\mathrm{ms}_{k,0}=P^\mathrm{ms}_kU_0$. The following lemma is a discrete versions of \cite[Lemma 3.3]{Thomee2006}.

\begin{mylemma}\label{erroreq1}
Suppose $e_n$ satisfies the error equation \eqref{erroreq}. Then 
\begin{align}
\|e_n\|^2 &\leq C\Big(\|\rho_{n}\|^2 + t_n^{-1}\Big( \sum_{j=1}^n\tau\|\rho_{j}\|^2 + \sum_{j=2}^n \tau t_j^2\|\bar\partial \rho_{j}\|^2\Big)\Big), \quad n \geq 2,\label{eq1:erroreq1} \\
\|e_1\| &\leq \|\rho_1\|. \label{eq2:erroreq1}
\end{align}
\end{mylemma}
\begin{proof}
Multiply the error equation \eqref{erroreq} by $\bar\partial_t e_n$ and integrate over $\Omega$ to get
\begin{align*}
(T^\mathrm{ms}_k \bar\partial_te_n,\bar\partial_t e_n) + (e_n, \bar\partial_t e_n) = (\rho_{n},\bar\partial_t e_n),
\end{align*}
where the first term on the left hand side is nonnegative, since $T^\mathrm{ms}_k$ is positive semidefinite on $L_2$. Multiplying by $\tau t_n$ we have
\begin{align*}
t_n\|e_n\|^2 - t_{n}(e_n,e_{n-1}) \leq t_n(\rho_{n}, e_n - e_{n-1}),
\end{align*}
which gives
\begin{align*}
\frac{t_n}{2}\|e_n\|^2 - \frac{t_{n-1}}{2}\|e_{n-1}\|^2 &\leq t_n(\rho_{n}, e_n-e_{n-1}) + \frac{t_n-t_{n-1}}{2}\|e_{n-1}\|^2 \\
&\leq t_n(\rho_{n}, e_n) - t_{n-1}(\rho_{n-1}, e_{n-1}) \\
&\quad - (t_n\rho_n-t_{n-1}\rho_{n-1},e_{n-1}) + \frac{\tau}{2}\|e_{n-1}\|^2.
\end{align*}
Summing over $n$ now gives
\begin{align*}
t_n\|e_n\|^2 - t_1\|e_1\|^2&\leq 2t_n(\rho_{n}, e_n) - 2t_1(\rho_{1}, e_1) - \sum_{j=2}^n2(t_j\rho_j-t_{j-1}\rho_{j-1},e_{j-1}) \\&\quad+ \sum_{j=2}^n\tau\|e_{j-1}\|^2,
\end{align*}
and thus,
\begin{align*}
t_n\|e_n\|^2 &\leq C\Big(t_n\|\rho_n\|^2 + \sum_{j=2}^n \tau \big(t_j^2\|\bar \partial_t \rho_j\|^2 + \|\rho_{j-1}\|^2\big) + \sum_{j=2}^n\tau\|e_{j-1}\|^2\Big).
\end{align*}
To estimate the last sum we note that, since $T^\mathrm{ms}_k$ is self-adjoint and positive semi-definte, 
\begin{align*}
2(T^\mathrm{ms}_k\bar\partial_te_n,e_n) &= (T^\mathrm{ms}_k\bar\partial_te_n,e_n)+ (T^\mathrm{ms}_ke_n, \bar \partial_t e_n)\\ &= \bar\partial_t (T^\mathrm{ms}_ke_n,e_n) + \tau(T^\mathrm{ms}_k\bar\partial_te_n,\bar\partial_te_n)
\geq \bar\partial_t(T^\mathrm{ms}_ke_n,e_n).
\end{align*}
so by multiplying the error equation \eqref{erroreq} by $2e_n$ we get
\begin{align*}
\bar\partial_t(T^\mathrm{ms}_ke_n,e_n) + 2\|e_n\|^2 \leq 2(T^\mathrm{ms}_k\bar\partial_te_n,e_n) + 2\|e_n\|^2 = 2(\rho_n,e_n).
\end{align*}
Multiplying by $\tau$ and summing over $n$ gives
\begin{align*}
\tau(T^\mathrm{ms}_ke_n,e_n) + \sum_{j=1}^n\tau\|e_j\|^2 \leq \sum_{j=1}^n \tau \|\rho_j\|^2,
\end{align*}
where we have used that $T^\mathrm{ms}_ke_0=0$. Since the first term is nonnegative we deduce that $\sum_{j=1}^n\tau\|e_j\|^2 \leq \sum_{j=1}^n \tau \|\rho_j\|^2$ and \eqref{eq1:erroreq1} follows. For $n=1$ this also proves \eqref{eq2:erroreq1}.
\end{proof}

Next Lemma is a discrete version of a result that can be found in the proof of \cite[Theorem~3.3]{Thomee2006}.

\begin{mylemma}\label{erroreq2}
Under the assumptions of Lemma~\ref{erroreq1} we have, for $n \geq 2$, the bound
\begin{align}\label{eq:erroreq2}
\|e_n\| \leq Ct_n^{-1} \Big(\max_{2\leq j\leq n}t_j^2\|\bar\partial_t \rho_j\| + \max_{1\leq j\leq n}\Big(t_j\|\rho_j\| +  \|\sum_{r=1}^j\tau\rho_r\|\Big)\Big).
\end{align}
\end{mylemma}

\begin{proof}
It follows from Lemma~\ref{erroreq1} that 
\begin{align*}
\|e_n\| \leq C(\max_{2\leq j \leq n}t_j\|\bar\partial_t\rho_j\| + \max_{1\leq j \leq n}\|\rho_j\|), \quad n \geq 2,
\end{align*}
or by using Young's inequality with different constants the proof can be modified to show that
\begin{align*}
\|e_n\| \leq \epsilon \max_{2\leq j \leq n}t_j\|\bar\partial_t\rho_j\| + C(\epsilon)\max_{1\leq j \leq n}\|\rho_j\|, \quad n \geq 2,
\end{align*}
for some $\epsilon >0$. Now define $z_j=t_je_j$. Then 
\begin{align*}
T^\mathrm{ms}_k\bar\partial_t z_n + z_n &= t_n\rho_n + T^\mathrm{ms}_ke_{n-1}:=\eta_n, \quad n \geq 1, 
\end{align*}
and, since $T^\mathrm{ms}_0 z_0 = 0$ we conclude from Lemma~\ref{erroreq1} 
\begin{align*}
\|z_n\| \leq \epsilon \max_{2\leq j \leq n}t_j\|\bar\partial_t\eta_j\| + C\max_{1\leq j \leq n}\|\eta_j\|.
\end{align*}
From the definition of $\eta_j$ it follows that
\begin{align*}
\|\eta_j\|&\leq t_j\|\rho_j\| + \|T^\mathrm{ms}_ke_{j-1}\|, \quad j \geq 1.
\end{align*}
Furthermore, for $j\geq 2$
\begin{align*}
t_j\|\bar\partial_t \eta_j\| &\leq t_j\|\bar\partial_tt_j\rho_j\| + t_j\|\bar\partial_t T^\mathrm{ms}_ke_{j-1}\|
\leq t_j^2 \|\bar\partial_t\rho_j\| +t_j \|\rho_{j-1}\| +  t_j\|\rho_{j-1}-e_{j-1}\| \\
&\leq t_j^2 \|\bar\partial_t\rho_j\| +  2t_j\|\rho_j-\rho_{j-1}\| + 2t_j \|\rho_{j}\| +  t_{j}\|e_{j-1}\| +  \\
&\leq 3t_j^2 \|\bar\partial_t\rho_j\| + 2t_j\|\rho_{j}\| +  2t_{j-1}\|e_{j-1}\| \\
&\leq C\Big(t_j^2\|\bar\partial_t\rho_j\| + t_j\|\rho_j\|\Big) + 2\|z_{j-1}\|,
\end{align*}
where we used $\frac{1}{2}t_{j}\leq t_{j-1} \leq t_j$ for $j\geq 2$.
To bound $\|T^\mathrm{ms}_ke_n\|$ we define $\tilde{e}_n = \sum_{j=1}^n\tau e_j$ and $\tilde e_0 =0$. Multiplying the error equation \eqref{erroreq} by $\tau$ and summing over $n$ gives
\begin{align*}
\sum_{j=1}^n \tau T^\mathrm{ms}_k \bar\partial_t e_n + \tilde{e}_n = T^\mathrm{ms}_k \bar\partial_t \tilde{e}_n + \tilde{e}_n = \tilde{\rho}_n, \quad n \geq 1,
\end{align*}
where $\tilde{\rho}_n=\sum_{j=1}^n \tau \rho_j$ and we have used that $T^\mathrm{ms}_ke_0=0$. Note that by definition $T^\mathrm{ms}_k\tilde e_0 = 0$. Thus, by Lemma~\ref{erroreq1}, we have
\begin{align*}
\|\tilde{e}_n\| &\leq C\Big(\max_{2\leq j \leq n}t_j\|\bar\partial_t\tilde{\rho}_j\| + \max_{1\leq j \leq n}\|\tilde{\rho}_j\|\Big)\leq C\Big(\max_{2\leq j \leq n}t_j\|\rho_j\| + \max_{1\leq j \leq n}\|\sum_{r=1}^j \tau\rho_r\|\Big).
\end{align*}
Hence, since $T^\mathrm{ms}_k \bar\partial_t \tilde{e}_n = T^\mathrm{ms}_k e_n$,
\begin{align*}
\|T^\mathrm{ms}_ke_n\| \leq \|\tilde e_n\| + \|\tilde \rho_n\| \leq C\Big(\max_{2\leq j \leq n}t_j\|\rho_j\| + \max_{1\leq j \leq n}\|\sum_{r=1}^j \tau\rho_r\|\Big).
\end{align*}
With $\epsilon = \frac{1}{4}$ we get
\begin{align*}
\|z_n\| &\leq \frac{1}{4} \max_{2\leq j \leq n}t_j\|\bar\partial_t\eta_j\| + C\max_{1\leq j \leq n}\|\eta_j\| \\
&\leq \frac{1}{2} \max_{1\leq j \leq n}\|z_{j}\| + C\Big(\max_{2\leq j \leq n}t_j^2\|\bar\partial_t\rho_j\| + \max_{1\leq j \leq n}(t_j\|\rho_j\| + \|\sum_{r=1}^j \tau\rho_r\|)\Big),
\end{align*}
but from \eqref{eq2:erroreq1} we deduce $\|z_1\|\leq t_1\|\rho_1\|$, and hence
\begin{align*}
\|z_n\| &\leq \frac{1}{2} \max_{2\leq j \leq n}\|z_{j}\| + C\Big(\max_{2\leq j \leq n}t_j^2\|\bar\partial_t\rho_j\| + \max_{1\leq j \leq n}(t_j\|\rho_j\| + \|\sum_{r=1}^j \tau\rho_r\|)\Big).
\end{align*}
Choosing $n^\ast$ such that $\max_{2\leq j \leq n}z_j = z_{n^\ast}$ we conclude \eqref{eq:erroreq2}.
\end{proof}

\begin{mylemma}\label{femconv1}
Assume $f=0$ and let $U^{\mathrm{ms}}_{k,n}$ be the solution to \eqref{gfemlod} and $U_n$ the solution to \eqref{femh}. Then, for $1\leq n \leq N$,
\begin{align*}
&\|U^{\mathrm{ms}}_{k,n}-U_{n}\| \leq  C(H+k^{d/2}\mu^{k})^2t_n^{-1}\|U_{0}\|
\end{align*}
where $C$ depends on $\beta$, $\alpha$, $\gamma$, and $T$, but not on the variations of $A$.
\end{mylemma} 
\begin{proof}
From Lemma~\ref{erroreq2} we have
\begin{align*}
\|e_n\| \leq Ct_n^{-1} \Big(\max_{2\leq j\leq n}t_j^2\|\bar\partial_t \rho_j\|+ \max_{1\leq j \leq n}\Big(t_j\|\rho_j\| + \|\sum_{r=1}^j\tau\rho_r\|\Big)\Big), \quad n \geq 2,
\end{align*}
and from Lemma~\ref{erroreq1} $\|e_1\| \leq \|\rho_1\|$.
The rest of the proof is based on estimates for the projection $R^\mathrm{ms}_k$ in Lemma~\ref{lodprojconv} and the regularity of the homogeneous equation \eqref{eq:femreg2}. 
We have
\begin{align*}
t_j^2\|\bar\partial_t{\rho}_j\| &\leq C(H+k^{d/2}\mu^k)^2 t_j^2\|\mathcal{A}_h\bar\partial_tU_j\|\\ &\leq C(H+k^{d/2}\mu^k)^2t_j^2\|\bar\partial_t\bar\partial_tU_j\| \leq C(H+k^{d/2}\mu^k)^2\|U_0\|, \quad j\geq 2,\\
t_j\|\rho_j(s)\| &\leq C(H+k^{d/2}\mu^k)^2 t_j\|\mathcal{A}_hU_j\| \leq C(H+k^{d/2}\mu^k)^2\|U_0\|, \quad j \geq 1, \\
\|\sum_{r=1}^j \tau\rho_r\| &= \|\sum_{r=1}^j\tau(T_h-T^\mathrm{ms}_k)\bar\partial_t{U}_r\| \leq\|\tau(T_h-T^\mathrm{ms}_k)(U_n-U_0)\| \\&\leq C(H+k^{d/2}\mu^k)^2\|U_0\|,
\end{align*}
where we have used $\|U_n\|\leq\|U_0\|$, which completes the proof.
\end{proof}

The next lemma concerns the convergence of the inhomogeneous parabolic problem $\eqref{parabolic}$ with initial data $U_{0}=0$. 

\begin{mylemma}\label{femconv2}
Assume $U_{0}=0$ and let $U^{\mathrm{ms}}_{k,n}$ be the solution to \eqref{gfemlod} and $U_n$ the solution to \eqref{femh}. Then, for $1\leq n \leq N$,
\begin{align*}
\|U^{\mathrm{ms}}_{k,n}-U_n\| &\leq  C(1 + \log\frac{t_n}{\tau})(H+k^{d/2}\mu^{k})^2(\|f\|_{L_\infty(L_2)}+\|\dot{f}\|_{L_\infty(L_2)}),
\end{align*}
where $C$ depends on $\beta$, $\alpha$, $\gamma$, and $T$, but not on the variations of $A$.
\end{mylemma}

\begin{proof}
Let $U^\mathrm{ms}_{k,n}- U_n = U^\mathrm{ms}_{k,n}-R^\mathrm{ms}_k U_n + R^\mathrm{ms}_k U_n - U_n=:\theta_{n} + \rho_{n}$. For $\rho_{n}$ we use Lemma~\ref{lodprojconv} to achieve the estimate 
\begin{align*}
\|\rho_{n}\|\leq C(H+k^{d/2}\mu^{k})^2\|\mathcal{A}_h U_n\|.
\end{align*}

Now, for $v \in V^{\mathrm{ms}}_k$ we have
\begin{align*}
\left(\bar \partial_t \theta_{n},v\right) + a(\theta_{n},v) =(-\bar\partial_t \rho_{n},v).
\end{align*}
Using Duhamel's principle we have
\begin{align*}
\theta_n = \tau \sum_{j=1}^n E^\mathrm{ms}_{k,n-j+1}P^\mathrm{ms}_k(-\bar\partial_t \rho_{j}),
\end{align*}
since $\theta_0=0$. Summation by parts now gives
\begin{align*}
\theta_n = E^\mathrm{ms}_{k,n}P^\mathrm{ms}_k\rho_0 - P^\mathrm{ms}_k\rho_n + \tau \sum_{j=1}^n \bar\partial_tE^\mathrm{ms}_{k,n-j+1}P^\mathrm{ms}_k\rho_{j}.
\end{align*}
Note that $\rho_0=0$. Using Lemma~\ref{lodprojconv} and Lemma~\ref{esthomogeneous_disc} we get
\begin{align*}
\|\theta_n\| &\leq  \|\rho_n\| + \tau \sum_{j=1}^n t_{n-j+1}^{-1}\|\rho_{j}\|\\&\leq C(H+k^{d/2}\mu^k)^2\max_{1\leq j\leq n}\|\mathcal{A}_hU_j\|(1+\tau\sum_{j=1}^nt_{n-j+1}^{-1}),
\end{align*}
where the last sum can be bounded by
\begin{align*}
\tau \sum_{j=1}^nt^{-1}_{n-j+1}\leq 1 + \log \frac{t_n}{\tau}.
\end{align*} 
It remains to bound $\|\mathcal{A}_h U_n\|$. We have $\mathcal{A}_hU_n=P_hf_n-\bar\partial_tU_n$ and Lemma~\ref{femreg} gives
\begin{align*}
\|\mathcal{A}_hU_j\| \leq \|f_j\| + \|\bar\partial_t U_j\| \leq C(\|f\|_{L_\infty(L_2)} + \|\dot{f}\|_{L_\infty(L_2)}),
\end{align*}
which completes the proof.
\end{proof}

\begin{proof}[Proof of Theorem~\ref{femconv}]
The result follows from Lemma~\ref{femconv1} and Lemma~\ref{femconv2} by rewriting $U_n=U_{n,1}+U_{n,2}$, where $U_{n,1}$ is the solution to the homogeneous problem and $U_{n,2}$ the solution to the inhomogeneous problem with vanishing initial data. 
\end{proof}

\begin{myremark}
We note that the choice of $k$ and the size of $\mu$ determine the rate of the convergence. In general, to achieve optimal order convergence rate, $k$ should be chosen proportional to $\log(H^{-1})$, i.e. $k =c \log(H^{-1})$. With this choice of $k$ we have $\|U^\mathrm{ms}_{k,n}-U_n\| \leq C(1+\log n)H^2t_n^{-1}$.
\end{myremark}

\section{The semilinear parabolic equation}\label{sec:semilin}
In this section we discuss how the above techniques can be extended to a semilinear parabolic problem with multiscale diffusion coefficient. 
\subsection{Problem formulation}
We are interested in equations of the form 
\begin{alignat}{2}
\dot{u} - \nabla \cdot (A\nabla u) &= f(u), & \quad  &\text{in } \Omega \times (0,T], \notag \\
u &= 0, & &\text{on } \partial \Omega \times (0,T], \label{semilin} \\
u(\cdot,0) &= u_0, & &\text{in } \Omega, \notag
\end{alignat}
where $f:\mathbb{R}\rightarrow \mathbb{R}$ is twice continuously differentiable and $\Omega$ is a polygonal/polyhedral boundary in $\mathbb{R}^d$, for $d\leq 3$. For $d=2,3$, $f$ is assumed to fulfill the growth condition 
\begin{align}\label{growth}
|f^{(l)}(\xi)|\leq C(1+|\xi|^{\delta + 1 -l}), \quad \text{for } l=1,2, 
\end{align}
where $\delta=2$ if $d=3$ and $\delta \in [1,\infty)$ if $d=2$. Furthermore, we assume that the diffusion $A$ fulfills assumption (A1) and $u_0\in V$.

\begin{myex}
The Allen-Cahn equation $\dot{u} - \nabla \cdot (A \nabla u) = -(u^3-u)$ fulfills the assumption \eqref{growth}.
\end{myex}

Define the ball $B_R := \{v \in V: \|v\|_{H^1} \leq R\}$. Using H\"{o}lder and Sobolev inequalities the following lemma can be proved, see \cite{larsson06}.
\begin{mylemma}\label{growthlemma}
If $f$ fulfills assumption \eqref{growth} and $u,v \in B_R$, then
\begin{align*}
\|f(u)\| \leq C, \
\|f'(u)z\|_{H^{-1}} \leq C\|z\|,\
\|f'(u)z\| \leq C\|z\|_{H^1},\
\|f''(u)z\|_{H^{-1}} \leq C\|z\|, \
\end{align*}
and
\begin{align*}
\|f(u)-f(v)\|_{H^{-1}} \leq C\|u-v\|,
\end{align*}
where C is a constant depending on $R$.
\end{mylemma}

From \eqref{semilin} we derive the variational form; find $u(t)\in V$ such that 
\begin{align}\label{semilinvar}
(\dot{u},v) + (A\nabla u,\nabla v) = (f(u),v), \quad \forall v \in V,
\end{align}
and $u(0)=u_0$. For this problem local existence of a solution can be derived given that the initial data $u_0 \in V$, see \cite{larsson06}.
\begin{mythm}\label{semilinreg}
Assume that (A1) and \eqref{growth} holds. Then, for $u_0\in B_R$, there exist $T_0 = T_0(R)$ and $c>0$, such that \eqref{semilinvar} has a unique solution $u \in C(0,T_0;V)$ and $\|u\|_{L_\infty(0,T_0;V)}\leq cR$. 
\end{mythm}

For the Allen-Cahn equation it is possible to find an a priori global bound of $u$. This means that for any time $T$ there exists $R$ such that if $u$ is a solution then $\|u(t)\|_{L_\infty(H^1)}\leq R$ for $t\in[0,T]$. Thus we can apply the local existence theorem repeatedly to attain global existence, see \cite{larsson06}. 

\subsection{Numerical approximation}
The assumptions and definitions of the families of triangulations $\{\mathcal{T}_h\}_{h>0}$ and $\{\mathcal{T}_H\}_{H>h}$ and the corresponding spaces $V_H$ and $V_h$ remain the same as in Section~\ref{sec:LOD}. For the discretization in time we use a uniform time discretization given by
\begin{align}\label{timedisc2}
0=t_0<t_1<...<t_N=T_0, \text{ where } t_n-t_{n-1} = \tau, 
\end{align}
where $T_0$ is given from Theorem~\ref{semilinreg}. With these discrete spaces we consider the semi-implicit backward Euler scheme where $U_n \in V_h$ satisfies
\begin{align}\label{semilinfemh}
(\bar\partial_t U_n,v) + (A\nabla U_n,\nabla v) = (f(U_{n-1}),v), \quad \forall v \in V_h,
\end{align}
for $n=1,...,N$ where $U_0 \in V_h$ is an approximation of $u_0$. It is proven in \cite{larsson92} that this scheme satisfies the bound 
\begin{align*}
\|U_n-u(t_n)\|\leq Ct_{n}^{-1/2}(h^2+\tau),
\end{align*}
if we choose, for instance, $U_0=P_hu_0$, where $P_h$ denotes the $L_2$-projection onto $V_h$. Note that $C$ in this bound depends on the variations of $A$.

The following theorem gives some regularity estimates of the solution to \eqref{semilinfemh}.

\begin{mythm}\label{semilinfemreg}
Assume that (A1) and \eqref{growth} holds. Then, for $U_0\in B_R$, there exist $T_0 = T_0(R)$ and $c>0$ such that \eqref{semilinfemh} has a unique solution $U_n \in V_h$, for $1\leq n\leq N$, and $\max_{1\leq n \leq N}\|U_n\|_{H^1}\leq cR$. Moreover, the following bounds hold 
\begin{align*}
\|\bar\partial_t U_n\|\leq Ct_n^{-1/2},\ n\geq 1, \quad \enorm{\bar\partial_t U_n}\leq Ct_n^{-1},\ n\geq 1, \quad  \|\bar\partial_t\bar\partial_t U_n\| \leq Ct_n^{-3/2},\ n\geq 2,
\end{align*}
where $C$ depends on $\alpha$, $T_0$, and $R$, but not on the variations of $A$.
\end{mythm}

\begin{proof}
We only prove the estimate $\|\bar\partial_t\bar\partial_t U_n\| \leq Ct_n^{-3/2}$ here. The other two follow by similar arguments. 

From \eqref{semilinfemh} we get
\begin{align}
(\bar\partial_t\bar\partial_t U_n,v) + a(\bar\partial_t U_n,v) &= (\bar\partial_t f(U_{n-1}),v),\quad \forall v \in V_h, \ n\geq 2, \label{diff1} \\
(\bar\partial_t^{(3)}U_n,v) + a(\bar\partial_t\bar\partial_t U_n,v) &= (\bar\partial_t\bar\partial_t f(U_{n-1}),v), \quad \forall v\in V_h, \ n \geq 3. \label{diff2}
\end{align}
Choosing $v = \bar\partial_t\bar\partial_t U_n$ in \eqref{diff2} gives
\begin{align*}
\frac{1}{\tau}\|\bar\partial_t\bar\partial_tU_n\|^2 - \frac{1}{\tau}(\bar\partial_t\bar\partial_t U_{n-1},\bar\partial_t\bar\partial_t U_n) + \enorm{\bar\partial_t\bar\partial_t U_n}^2= (\bar\partial_t\bar\partial_t f(U_{n-1}),\bar\partial_t\bar\partial_t U_{n}),
\end{align*}
which gives the bound
\begin{align}\label{bound2}
\|\bar\partial_t\bar\partial_tU_n\|^2 - \|\bar\partial_t\bar\partial_t U_{n-1}\|^2 \leq C\tau\|\bar\partial_t\bar\partial_t f(U_{n-1})\|_{H^{-1}}.
\end{align}
Using Lemma~\ref{growthlemma} we have for $\xi_j \in (\min\{U_{n-j},U_{n-(j-1)}\},\max\{U_{n-j},U_{n-(j-1)}\})$
\begin{align*}
\|\bar\partial_t\bar\partial_t f(U_n)\|_{H^{-1}} &= \frac{1}{\tau^2}\|f'(\xi_1)(U_{n}-U_{n-1}) - f'(\xi_2)(U_{n-1}-U_{n-2})\|_{H^{-1}}\\
&\leq\frac{1}{\tau^2}\|(f'(\xi_1)-f'(\xi_2))(U_{n}-U_{n-1})\|_{H^{-1}} \\& \quad+\frac{1}{\tau^2}\|f'(\xi_2)(U_{n}-2U_{n-1}+U_{n-2})\|_{H^{-1}} \\&\leq \frac{1}{\tau^2} \|(\xi_1-\xi_2)(U_{n}-U_{n-1})\| + C\|\bar\partial_t\bar\partial_t U_{n}\|,
\end{align*}
Note that $|\xi_1-\xi_2\|\leq \|U_{n-2}-U_{n-1}\| + \|U_{n-1}-U_{n}\|$. By using Sobolev embeddings we get
\begin{align*}
\frac{1}{\tau^2} \|(\xi_1-\xi_2)(U_{n}-U_{n-1})\| &\leq \max_{n-1\leq j \leq n} 2\|(\bar\partial_tU_{j})^2\| \leq \max_{n-1\leq j \leq n} 2\|\bar\partial_t U_{j}\|^2_{L_4} \\&\leq C\max_{n-1\leq j \leq n}\|\bar\partial_tU_{j}\|^2_{H^1} \leq Ct^2_{n-1} \leq Ct^2_{n},
\end{align*}
where we recall the bounds $\frac{1}{2}t_j\leq t_{j-1}\leq t_j$ for $j\geq 2$. Multiplying by $\tau t_n^4$ in \eqref{bound2} and summing over $n$ gives
\begin{align*}
t_n^4\|\bar\partial_t\bar\partial_tU_n\|^2 &\leq  t^4_2\|\bar\partial_t\bar\partial_tU_2\|^2 + \sum_{j=3}^n(\tau t_j^4\|\bar\partial_t\bar\partial_t f(U_{n-1})\|^2_{H^{-1}} + (t_j^4-t_{j-1}^4)\|\bar\partial_t\bar\partial_tU_{j-1}\|^2)\\
&\leq t^4_2\|\bar\partial_t\bar\partial_tU_2\|^2 + C \sum_{j=3}^n \tau\big(t_j^4\|\bar\partial_t\bar\partial_tU_{j-1}\|^2 +  t_j^4t^{-4}_{j-1} +  t_{j-1}^3\|\bar\partial_t\bar\partial_tU_{j-1}\|^2\big) \\
& \leq t^4_2\|\bar\partial_t\bar\partial_tU_2\|^2 + Ct_n + C\sum_{j=3}^n \tau \big(t_{j-1}^4\|\bar\partial_t\bar\partial_tU_{j-1}\|^2 +t_{j-1}^3\|\bar\partial_t\bar\partial_tU_{j-1}\|^2\big),
\end{align*}
for $n \geq 3$. Using $\|\bar\partial_tU_j\|\leq Ct_j^{-1/2}$ for $j\geq 1$ we get
\begin{align*}
 t^4_2\|\bar\partial_t\bar\partial_tU_2\|^2 \leq C\tau^2(\|\bar\partial_tU_2\|^2+\|\bar\partial_tU_1\|^2) \leq C \tau^2(t_2^{-1} + t_1^{-1}) \leq C\tau.
\end{align*}
Now, to bound $\sum_{j=2}^n t_j^3\|\bar\partial_t\bar\partial_tU_j\|$, we choose $v=\bar\partial_t\bar\partial_t U_n$ in \eqref{diff1} to derive
\begin{align}
\|\bar\partial_t\bar\partial_t U_n\|^2 + \frac{1}{\tau}\enorm{\bar\partial_t U_n}^2 - \frac{1}{\tau}\enorm{\bar\partial_t U_{n-1}}^2 &\leq \|\bar\partial_t f(U_{n-1})\|^2. \label{bound1}
\end{align}
and with $\xi_j$ as above, we get
\begin{align*}
\|\bar\partial_t f(U_{n-1})\| = \|f'(\xi_2)\bar\partial_t U_{n-1} \| \leq C\enorm{\bar\partial_t U_{n-1}} \leq Ct^{-1}_{n-1},
\end{align*}
where we used Lemma~\ref{growthlemma} and $\enorm{\bar\partial_t U_j} \leq Ct_{j}^{-1}$ for $j \geq 1$. Multiplying \eqref{bound1} with $\tau t_n^3$ and summing over $n$ gives
\begin{align*}
\sum_{j=2}^n\tau t_j^3\|\bar\partial_t\bar\partial_t U_j\|^2 + t_n^3\enorm{\bar\partial_t U_n}^2 &\leq C \sum_{j=2}^n (\tau t^3_jt^{-2}_{j-1} + (t^3_j-t^3_{j-1})\enorm{\bar\partial_t U_{j-1}}^2) \\ & \quad + t^3_{1}\enorm{\bar\partial_t U_{1}}^2 \\ 
&\leq C \sum_{j=2}^n (\tau t_j + \tau t^2_{j-1}\enorm{\bar\partial_t U_{j-1}}^2) + t^3_{1}\enorm{\bar\partial_t U_{1}}^2 . 
\end{align*}
Using $\enorm{\bar\partial_t U_j}\leq Ct_j^{-1}$ for $j\geq 1$ we get
\begin{align*}
\sum_{j=2}^n\tau t_j^3\|\bar\partial_t\bar\partial_t U_j\|^2 \leq C(t_n^2 + t_n + t_1) \leq Ct_n, 
\end{align*}
where $C$ now depends on $t_n\leq T$. So we have proved
\begin{align*}
t_n^4\|\bar\partial_t\bar\partial_tU_n\|^2 & \leq C\sum_{j=3}^n \tau t_{j-1}^4\|\bar\partial_t\bar\partial_tU_{j-1}\|^2 + Ct_n + \tau \\ &\leq C\sum_{j=2}^{n-1} \tau t_{j-1}^4\|\bar\partial_t\bar\partial_tU_{j-1}\|^2 + Ct_{n+1} \leq C\sum_{j=2}^{n-1} \tau t_{j-1}^4\|\bar\partial_t\bar\partial_tU_{j-1}\|^2 + Ct_n.
\end{align*}
Applying the classical discrete Gr\"{o}nwall's lemma gives 
\begin{align*}
t_n^4\|\bar\partial_t\bar\partial_tU_j\|^2 \leq Ct_n,
\end{align*}
which proves $\|\bar\partial_t\bar\partial_tU_j\|\leq Ct_n^{-3/2}$ for $n\geq 3$. For $n=2$ we proved 
\begin{align*}
t^4_2\|\bar\partial_t\bar\partial_tU_2\|^2 \leq C\tau \leq Ct_2,
\end{align*}
which completes the proof.
\end{proof}

We use the same GFEM space as in Section~\ref{sec:LOD}, that is, $V^\mathrm{ms} = V_H - R^\mathrm{f}(V_H)$ and the localized version $V^\mathrm{ms}_k = V_H - R^\mathrm{f}_k(V_H)$. Furthermore, for the completely discrete scheme, we consider the time discretization defined in \eqref{timedisc2} and the linearized backward Euler method thus reads; find $U^\mathrm{ms}_{k,n} \in V^\mathrm{ms}$ such that $U^\mathrm{ms}_{k,0} = P^\mathrm{ms}_k U_0$ and 
\begin{align}\label{semilinfully}
(\bar\partial_tU^\mathrm{ms}_{k,n}, v) + a(U^\mathrm{ms}_{k,n},v) = (f(U^\mathrm{ms}_{k,n-1}),v), 
\end{align}
for $n=1,...,N$ where $P^\mathrm{ms}_k$ is the $L_2$-projection onto $V^\mathrm{ms}_k$.

To derive an error estimates we represent the solution to \eqref{semilinfully} by using Duhamel's principle. Note that $U^\mathrm{ms}_{k,n}$ is the solution to the equation
\begin{align*}
\bar{\partial}_t U^\mathrm{ms}_{k,n} + \mathcal{A}^\mathrm{ms}_kU^\mathrm{ms}_{k,n} = P^\mathrm{ms}_kf(U^\mathrm{ms}_{k,n-1}),
\end{align*}
and by Duhamel's principle we get
\begin{align*}
U^\mathrm{ms}_{k,n} = E^\mathrm{ms}_{k,n}U^\mathrm{ms}_{k,0} + \tau\sum_{j=1}^n E^\mathrm{ms}_{k,n-j-1}P^\mathrm{ms}_kf(U^\mathrm{ms}_{k,j-1}).
\end{align*}

\subsection{Error analysis}\label{sec:errorsemilinfully}
For the error analysis we need the following generalized discrete Gr\"{o}nwall lemma, see, e.g., \cite{larsson06}.
\begin{mylemma}\label{gronwall}
Let $A,B\geq 0$, $\gamma_1,\gamma_2 >0$, $0\leq t_0<t_n\leq T$, and $0\leq \varphi_n \leq R$. If
\begin{align*}
\varphi_n \leq At_n^{-1+\gamma_1} + B\tau\sum_{j=1}^{n-1} t_{n-j+1}^{-1+\gamma_2}\varphi_j,
\end{align*}
then there is a constant $C$ depending on $B$, $\gamma_1$, $\gamma_2$, and, $T$, such that,
\begin{align*}
\varphi_n \leq At_n^{-1+\gamma_1}.
\end{align*} 
\end{mylemma} 

\begin{mythm}\label{fullyconv}
For given $R\geq 0$ and $T_0 >0$ let $U_n$ be the solution to \eqref{semilinfemh} and $U^\mathrm{ms}_{k,n}$ be the solution to \eqref{semilinfully}, such that $U_n, U^\mathrm{ms}_{k,n} \in B_R$. Then, for $1\leq n \leq N$,
\begin{align}\label{eq:fullyconv}
\|U^\mathrm{ms}_{k,n}-U_n\| \leq C((H+k^{d/2}\mu^k)^2+\tau)t_n^{-1/2},
\end{align}
where $C$ depends on $\beta$, $\alpha$, $\gamma$, $R$, and $T_0$, but not on the variations of $A$.
\end{mythm}
\begin{proof}
First we define $e_n = U^\mathrm{ms}_{k,n} - U_n =(U^\mathrm{ms}_{k,n} - R^\mathrm{ms}_k U_n) + (R^\mathrm{ms}_kU_n - U_n)=\theta_{n} + \rho_{n}$. For $\rho_{j}$ we use Lemma~\ref{lodprojconv} to prove the bounds
\begin{align*}
\|\rho_{j}\|&\leq C(H+k^{d/2}\mu^k)^2t^{-1/2}_{j}, \quad j \geq 1,
\end{align*}
and
\begin{align*}
\|\bar\partial_t \rho_{k,j}\| \leq C(H+k^{d/2}\mu^k)^2t^{-3/2}_{j}, \quad j\geq 2.
\end{align*}
For $\theta_{n}$ we have
\begin{align*}
\theta_{n} = E^\mathrm{ms}_{k,n}\theta_{0} + \tau \sum_{j=1}^n E^\mathrm{ms}_{k,n-j+1}P^\mathrm{ms}_k(f(U^\mathrm{ms}_{k,j-1})-f(U_{j-1})-\bar{\partial}_t\rho_{j}).
\end{align*} 

To bound $\|\theta_{k,n}\|$ we first assume $n\geq 2$ and use summation by parts for the first part of the sum. Defining $n_2$ to be the integer part of $n/2$ we can write
\begin{align*}
- \tau \sum_{j=1}^{n_2}E^\mathrm{ms}_{k,n-j+1}P^\mathrm{ms}_k\bar\partial_t\rho_j &= E^\mathrm{ms}_{k,n}P^\mathrm{ms}_k\rho_{0} - E^\mathrm{ms}_{k,n-n_2}P^\mathrm{ms}_k\rho_{n_2} \\ \quad&+ \tau \sum_{j=1}^{n_2}\big(\bar{\partial}_tE^\mathrm{ms}_{n-j+1}\big)P^\mathrm{ms}_k\rho_{j},
\end{align*} 
and $\theta_{n}$ can be rewritten as
\begin{align*}
\theta_{n} &= E^\mathrm{ms}_{k,n}P^\mathrm{ms}_ke_0 - E^\mathrm{ms}_{k,n-n_2}P^\mathrm{ms}_k\rho_{n_2} + \tau \sum_{j=1}^{n_2}\big(\bar{\partial}_tE^\mathrm{ms}_{k,n-j+1}\big)P^\mathrm{ms}_k\rho_{j} \\
&\quad  -\tau\sum_{j=n_2+1}^nE^\mathrm{ms}_{n-j+1}P^\mathrm{ms}_k\bar{\partial}_t\rho_{j} \\
&\quad + \tau\sum_{j=1}^n(\mathcal{A}^\mathrm{ms}_k)^{1/2}E^\mathrm{ms}_{k,n-j+1}(\mathcal{A}^\mathrm{ms}_k)^{-1/2}P^\mathrm{ms}_k(f(U^\mathrm{ms}_{k,j-1})-f(U_{j-1})),
\end{align*}
where we note that $P^\mathrm{ms}_ke_0=0$. To estimate these terms we need the following bounds for $\gamma_1,\gamma_2>0$
\begin{align*}
\tau \sum_{j=1}^nt_{n-j+1}^{-1+\gamma_1}t_j^{-1+\gamma_2} \leq C_{\gamma_1,\gamma_2}t_n^{-1+\gamma_1 + \gamma_2}, \quad
\tau \sum_{j=1}^{n_2}t_{n-j+1}^{-\gamma_1}t_j^{-1+\gamma_2} \leq C_{\gamma_1,\gamma_2}t_n^{-\gamma_1 + \gamma_2}.
\end{align*}
see \cite{larsson92}. Using Lemma~\ref{esthomogeneous_disc} we get
\begin{align*}
 \| \theta_n\| 
&\leq \|\rho_{k,n_2}\| + C\tau \sum_{j=1}^{n_2}t^{-1}_{n-j+1}\|\rho_{k,j}\| + C\tau \sum_{j=n_2+1}^n\|\bar\partial_t\rho_{k,j}\| \\ &\quad+ C\tau \sum_{j=1}^n t^{-1/2}_{n-j+1}\|f(U^\mathrm{ms}_{k,j-1})-f(U_{j-1})\|_{H^{-1}},
\end{align*}
and together with Lemma~\ref{lodprojconv} and Lemma~\ref{growthlemma} this gives
\begin{align*}
\|\theta_n\|&\leq C(H+k^{d/2}\mu^{k})^2\Big(t_{n_2}^{-1/2} + \tau\sum_{j=1}^{n_2}t^{-1}_{n-j+1}t_{j}^{-1/2} + \tau \sum_{j=n_2+1}^nt_j^{-3/2}\Big) \\
&\quad+C\tau \sum_{j=1}^n t^{-1/2}_{n-j+1}\|U^\mathrm{ms}_{k,j-1}-U_{j-1}\|\\
&\leq C(H+k^{d/2}\mu^{k})^2t_n^{-1/2} + C\tau\sum_{j=1}^nt^{-1/2}_{n-j+1}\|e_{j-1}\|.
\end{align*}
Now consider $\theta_{1}$. We can rewrite 
\begin{align*}
\theta_{1} &= E^\mathrm{ms}_{k,1}\theta_{0} + \tau E^\mathrm{ms}_{k,1}P^\mathrm{ms}_k(f(U^\mathrm{ms}_{k,0})-f(U_1)-\bar{\partial}_t\rho_{1}) \\
&=E^\mathrm{ms}_{k,1}P^\mathrm{ms}_ke_{0} + E^\mathrm{ms}_{k,1}P^\mathrm{ms}_k\rho_{1} + \tau E^\mathrm{ms}_{k,1}P^\mathrm{ms}_k(f(U^\mathrm{ms}_{k,0})-f(U_1)),
\end{align*}
and using similar arguments as above
\begin{align*}
\|\theta_{1}\| &\leq C(H+k^{d/2}\mu^k)^2t_1^{-1/2} + \tau t_1^{-1/2}\|e_0\|,
\end{align*}

Hence, we arrive at the estimate
\begin{align*}
\|e_n\| \leq Ct_n^{-1/2}(H+k^{d/2}\mu)^2 + C\tau \sum_{j=1}^nt^{-1/2}_{n-j+1}\|e_{j-1}\|, \quad n \geq 1,
\end{align*}
and we can use Lemma~\ref{gronwall} to conclude \eqref{eq:fullyconv}.
\end{proof}

\section{Numerical Results}\label{sec:numerics}
In this section we present some numerical results to verify the predicted error estimates presented for the linear problem in Section~\ref{sec:error} and the semilinear problem in Section~\ref{sec:semilin}. In both cases the domain is set to the unit square $\Omega=[0,1]\times [0,1]$ and $T = 1$. The domain $\Omega$ is discretized with a uniform triangulation and the interval $[0,T]$ is divided into subintervals of equal length.

The method is tested on two different types of diffusion coefficients $A_1$ and $A_2$ defined as
\begin{align*}
A_1(x) =  \begin{pmatrix}
  1 & 0 \\
  0 & 1
 \end{pmatrix}, \quad
A_2(x) =
 \begin{pmatrix}
  B(x) & 0 \\
  0 & B(x)
 \end{pmatrix},
\end{align*}
where $B$ is piecewise constant with respect to a uniform Cartesian grid of size $2^{-6}$, see Figure~\ref{coefficients}. Note that our choice of $B$ imposes significant multiscale behavior on the diffusion coefficient. Here we expect quadratic convergence in space of the standard finite element with piecewise linear and continuous polynomials (P1-FEM) when $A=A_1$, but poor convergence when $A=A_2$. For the GFEM we expect quadratic convergence in both cases.  

\begin{figure}[h]
\centering
\begin{subfigure}[b]{0.48\textwidth}
    \includegraphics[width=\textwidth]{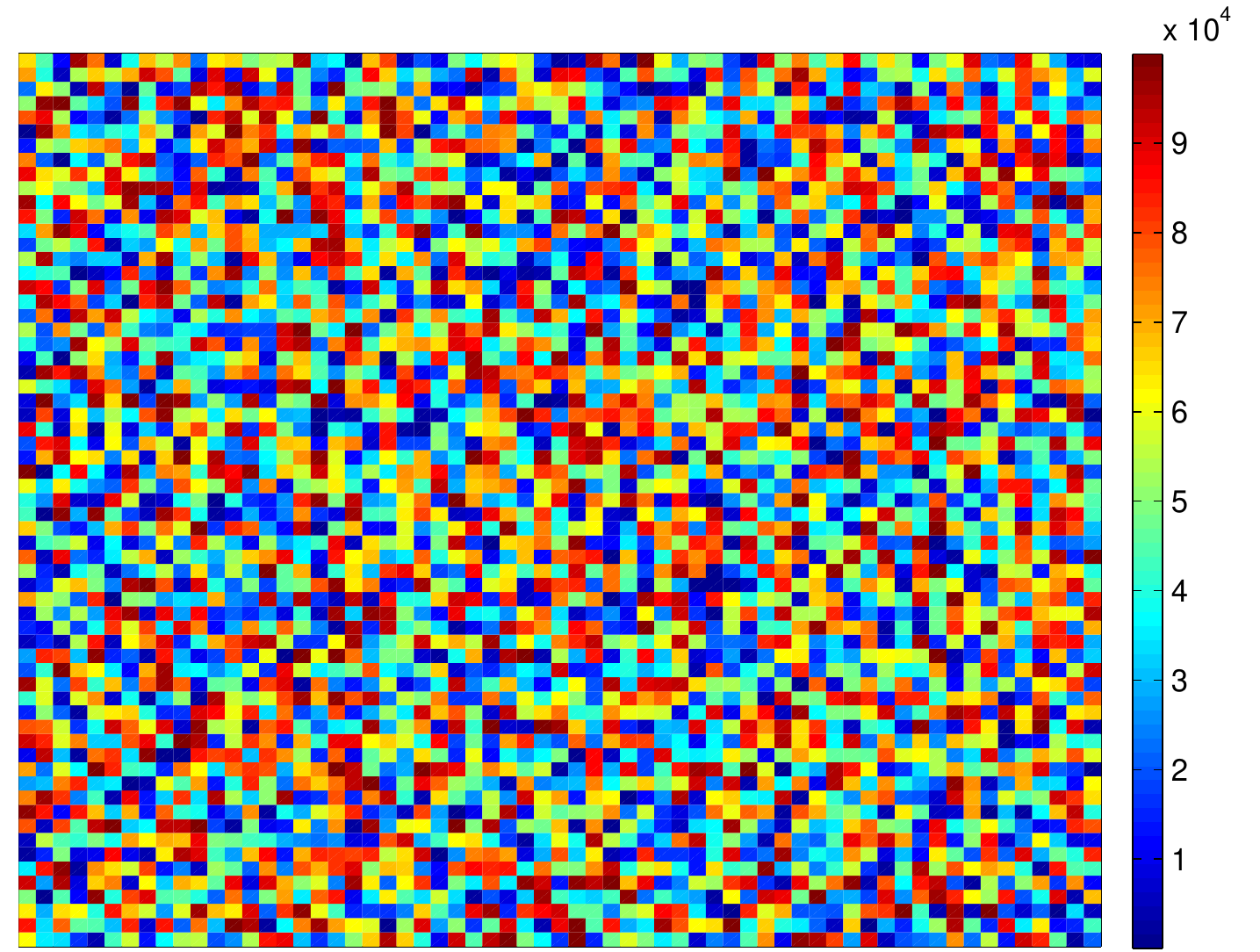}
   \caption{Coefficient $B$ for the linear parabolic problem. The contrast is $\beta/\alpha \approx 10^6$.}
\end{subfigure}
~
\begin{subfigure}[b]{0.48\textwidth}
  \centering
    \includegraphics[width=\textwidth]{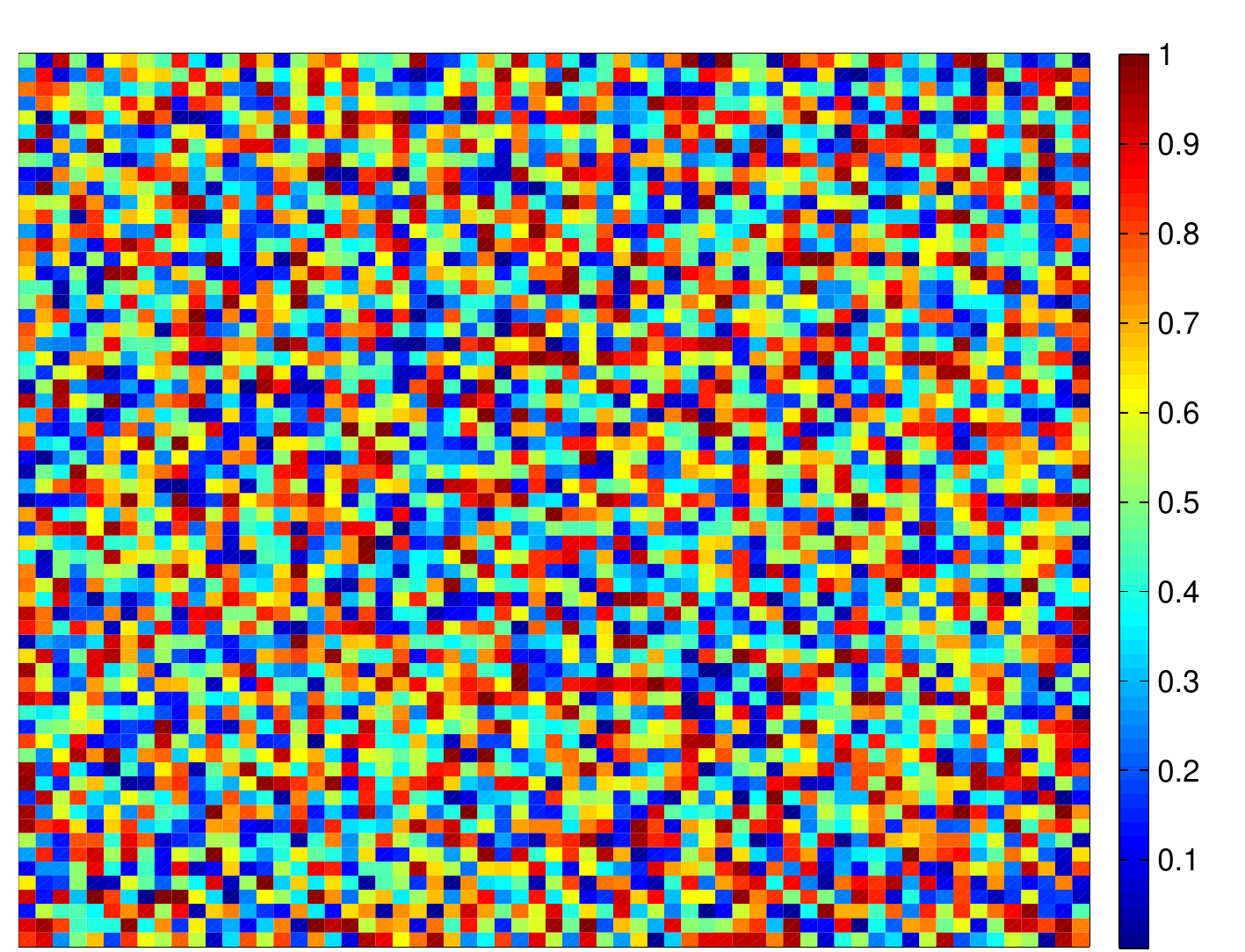}
  \caption{Coefficient $B$ for the semilinear parabolic problem. The contrast is $\beta/\alpha \approx 10^3$.}
\end{subfigure}
\caption{Coefficients for the two parabolic problems.}\label{coefficients}
\end{figure}

We compute the localized GFEM in \eqref{gfemlod} and \eqref{semilinfully}, denoted $U^\mathrm{ms}_{k,n}$, for 5 different values of the coarse grid width, $H= \sqrt{2}\cdot2^{-2}, \sqrt{2}\cdot2^{-3}, \sqrt{2}\cdot2^{-4}, \sqrt{2}\cdot2^{-5},$ and $\sqrt{2}\cdot2^{-6}$. The time step is chosen to $\tau=0.01$ for all problems. The reference mesh $\mathcal{T}_h$ is of size $h=\sqrt{2}\cdot 2^{-7}$ and defines the space $V_h$ on which the localized corrector problems $\phi_{k,x}$ are solved. To measure the error, the solution $U_{n}$ in \eqref{femh} is computed using P1-FEM on the finest scale $h=\sqrt{2}\cdot2^{-7}$ with $\tau=0.01$. 

Note that this experiment measures the error $\|U_{n}-U^\mathrm{ms}_{k,n}\|$. The total error $\|u(t_n)-U^\mathrm{ms}_{k,n}\|$ is also affected by the difference $\|u(t_n)-U_{n}\|$, which is dominating for the smaller values of $H$. We now present the result in two separate sections.

\subsection{Linear parabolic problem}

For the linear parabolic problem \eqref{parabolic} the right hand side is set to $f(x,t)=t$, which fulfills the assumptions for the required regularity. For simplicity the initial data is set to $u_0=1$. Moreover, at each cell in the Cartesian grid we choose a value from the interval $[10^{-1},10^5]$. This procedure gives $B$ a rapidly varying feature and a high contrast $\max(B)/\min(B) \approx 10^6$, see Figure~\ref{coefficients} (left). 

For each value of $H$ the localized GFEM, $U^\mathrm{ms}_{k,n}$, and the corresponding P1-FEM, denoted $U_{H,n}$, are computed. The patch sizes $k$ are chosen such that $k \sim \log(H^{-1})$, that is $k=1,2,2,3,$ and $4$ for the five simulations. When computing $U_{H,n}$ the stiffness matrix is assembled on the fine scale $h$ and then interpolated to the coarser scale. This way we avoid quadrature errors. The convergence results for $A_1$ and $A_2$ are presented in Figure~\ref{plotlinear}, where the error at the final time $t_N$ is plotted against the degrees of freedom $|\mathcal{N}|$. Comparing the plots we can see the predicted quadratic convergence for the localized GFEM. However, as expected, the P1-FEM shows poor convergence on the coarse grids when the diffusion coefficient has multiscale features. We clearly see the pre-asymptotic effects when $H$ does not resolve the fine structure of $B$. 

\begin{figure}[h]
  \centering
\begin{subfigure}[b]{0.48\textwidth}
    \includegraphics[width=\textwidth]{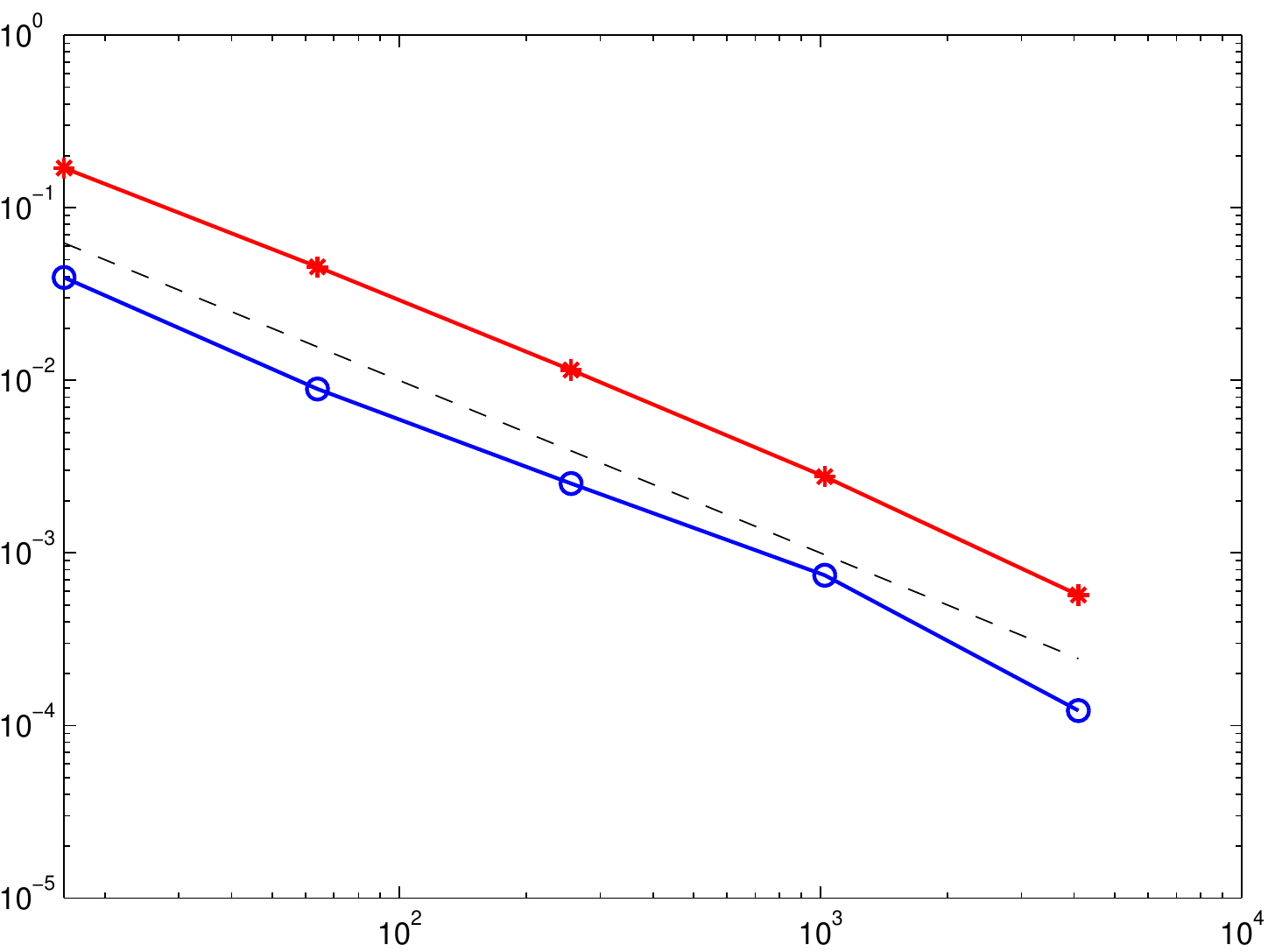}
   \caption{Constant coefficient $A_1$.}\label{plotconst}
\end{subfigure}
~
\begin{subfigure}[b]{0.48\textwidth}
  \centering
    \includegraphics[width=\textwidth]{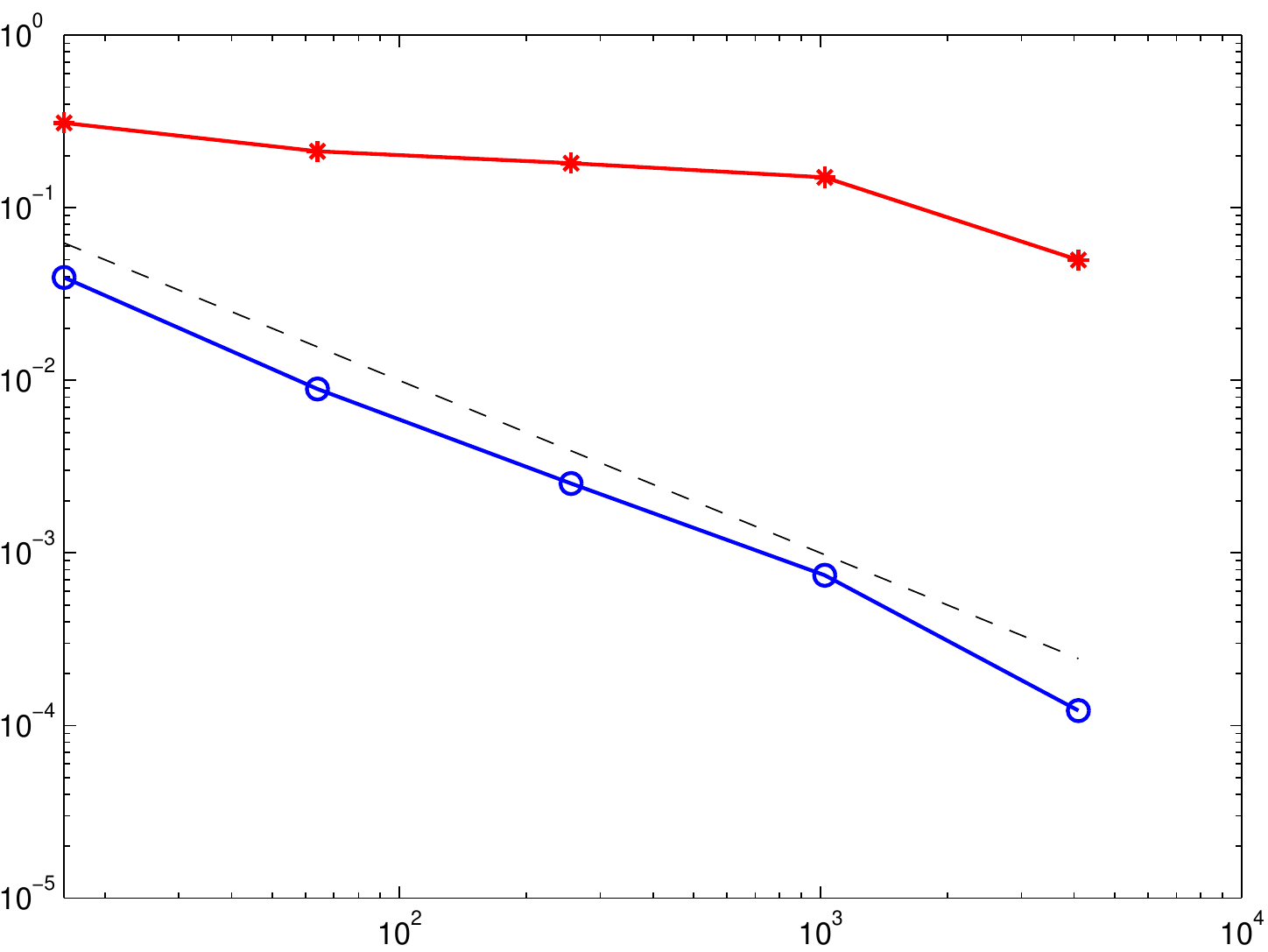}
  \caption{Multiscale coefficient $A_2$.}\label{plotms} 
\end{subfigure}
\caption{Relative $L_2$ errors $\|U^\mathrm{ms}_{k,N}-U_{h,N}\|/\|U_{h,N}\|$ (blue $\circ$) and $\|U_{H,N}-U_{h,N}\|/\|U_{h,N}\|$ (red $\ast$) for the linear parabolic problem plotted against the number of degrees of freedom $|\mathcal{N}|\approx H^{-2}$. The dashed line is $H^2$.}\label{plotlinear}
\end{figure}

\subsection{Semilinear parabolic problem}

For the semilinear problem we study the Allen-Cahn equation, which has right hand side $f(u) = -(u^3-u)$ that fulfills the necessary assumptions. We define the initial data to be $u_0(x,y)=x(1-x)y(1-y)$, which is zero on $\partial \Omega$. The matrix $B$ constructed as in the linear case but with values varying between $10^{-3}$ and $1$. Note that the solution to the Allen-Cahn equation converges to zero rapidly if the diffusion is too high, thus the smaller contrast $\max(B)/\min(B)\approx10^3$ in this case, see Figure~\ref{coefficients} (right). However, $B$ is still rapidly varying. As in the linear case we now compute the localized GFEM approximations $U^\mathrm{ms}_{k,n}$ and the corresponding P1-FEM, $U_{H,n}$. The patch sizes are chosen to $k=1,2,2,3,$ and $4$, for the five simulations. The convergence results for $A_1$ and $A_2$ are presented in Figure~\ref{plotsemilin}. We can draw the same conclusions as in the linear case. The localized GFEM shows predicted quadratic convergence in both cases, but P1-FEM shows poor convergence on the coarse grids when the diffusion coefficient has multiscale features. 

\begin{figure}[h]
  \centering
\begin{subfigure}[b]{0.48\textwidth}
    \includegraphics[width=\textwidth]{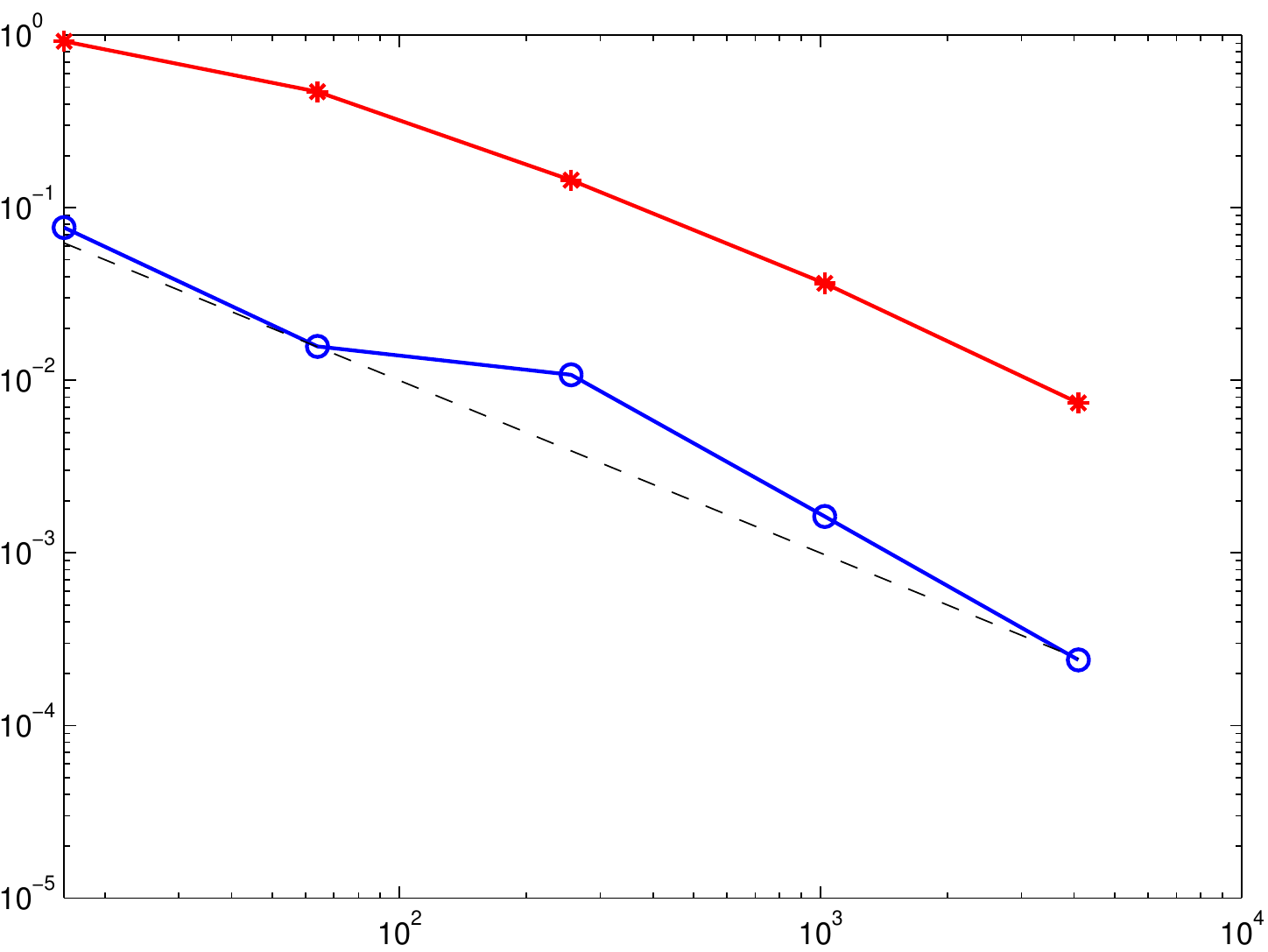}
   \caption{Constant coefficient $A_1$.}\label{plotconstsemilin}
\end{subfigure}
~
\begin{subfigure}[b]{0.48\textwidth}
  \centering
    \includegraphics[width=\textwidth]{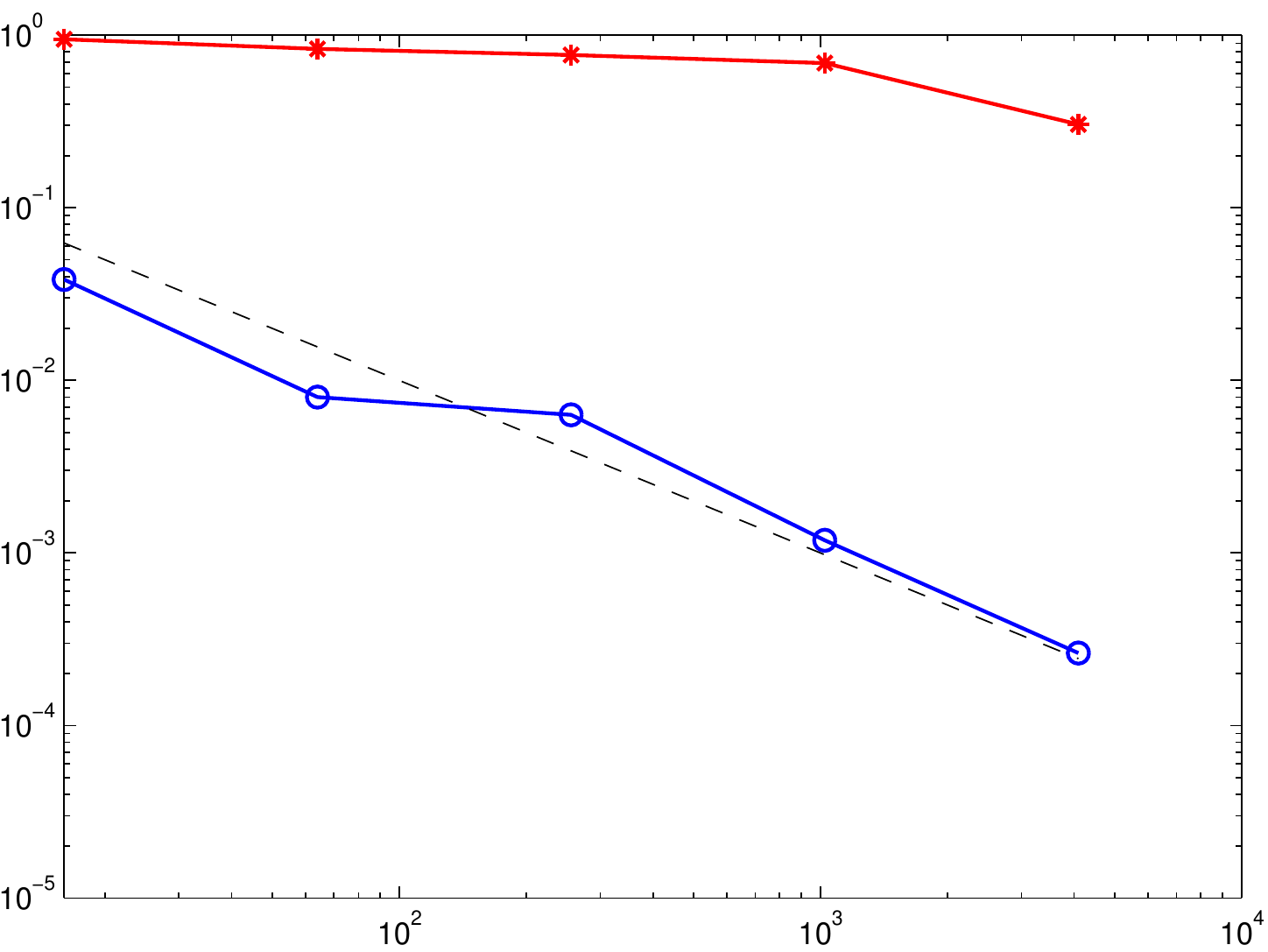}
  \caption{Multiscale coefficient $A_2$.}\label{plotmssemilin} 
\end{subfigure}
\caption{Relative $L_2$ errors $\|U^\mathrm{ms}_{k,N}-U_{h,N}\|/\|U_{h,N}\|$ (blue $\circ$) and $\|U_{H,N}-U_{h,N}\|/\|U_{h,N}\|$ (red $\ast$) for the semilinear parabolic problem plotted against the number of degrees of freedom $|\mathcal{N}|\approx H^{-2}$. The dashed line is $H^2$.}\label{plotsemilin}
\end{figure}

\bibliographystyle{plain}
\bibliography{parabolic_ref}
\end{document}